\documentclass[preprint,10pt]{elsarticle}

\usepackage{diagrams}
\usepackage{amssymb,amsmath}

\def\Q{Q}
\def\wt{\widetilde}

\def\Jac{\mbox{\rm Jac}\:}

\def\Hom{\mbox{Hom}}
\def\rank{\mbox{ rank}}
\def\Ext{\mbox{Ext}}

\def\End{E}
\def\l{\langle\:}
\def\r{\:\rangle}

\def\wt{\widetilde}
\def\QQ{{Q}}

\def\Hom{\mbox{Hom}}
\def\rank{\mbox{ rank}}
\def\Ext{\mbox{Ext}}

\def\End{E}
\def\l{\langle\:}
\def\r{\:\rangle}

 \usepackage{amsthm}

\begin{document}

\newtheorem{theorem}{Theorem}[section]
\newtheorem{lemma}[theorem]{Lemma}
\newtheorem{proposition}[theorem]{Proposition}
\newtheorem{corollary}[theorem]{Corollary}
\newtheorem{problem}[theorem]{Problem}
\newtheorem{question}[theorem]{Question}
\newtheorem{example}[theorem]{Example}
\newtheorem{setting}[theorem]{Setting}

\newtheorem{defi}[theorem]{Definitions}
\newtheorem{definition}[theorem]{Definition}
\newtheorem{definitions}[theorem]{Definitions}
\newtheorem{remark}[theorem]{Remark}
\newtheorem{remarks}[theorem]{Remarks}
\newtheorem{ex}[theorem]{Example}
\newtheorem{notation}[theorem]{Notation}

\begin{frontmatter}



\title{One-dimensional stable rings}



\author{Bruce Olberding}

\address{Department of Mathematical Sciences, New Mexico State University,
Las Cruces, NM 88003-8001}

\ead{olberdin@nmsu.edu}

\begin{abstract}
A commutative ring $R$ is  stable provided every   ideal of $R$ containing a nonzerodivisor  is projective as a module over its ring of endomorphisms. 
 The class of stable rings includes the  one-dimensional local Cohen-Macaulay rings of multiplicity at most $2$, as well as  certain rings of higher multiplicity, necessarily analytically ramified. The former are important in the study of modules over Gorenstein rings, while the latter arise in a natural way from generic formal fibers and derivations. 
 
 We characterize one-dimensional   stable local rings in several ways. The characterizations involve the integral closure ${\overline{R}}$ of $R$ and the completion  of $R$ in a relevant ideal-adic topology. For example, we show: If $R$ is a reduced stable ring, then there are exactly two possibilities  for $R$: (1) $R$ is a {\it Bass ring}, that is, $R$ is a reduced Noetherian local ring such that $\overline{R}$ is finitely generated over $R$ and every ideal of $R$ is generated by two elements; or (2) $R$ is a {\it bad stable domain}, that is, $R$ is a one-dimensional stable local domain such that $\overline{R}$ is not a finitely generated $R$-module.

\end{abstract}


\end{frontmatter}



\section{Introduction}

  The class of stable rings has a long history dating back at least to the ``Ubiquity'' paper of Bass, where he showed that rings for which every ideal can be generated by two elements are stable \cite[Corollary 7.3]{Bass}.  
Following Lipman \cite{Lipman} and Sally and Vasconcelos \cite{SVBull, SV}, we define an ideal $I$ of a commutative ring $R$ to be {\it stable} if $I$ is projective over its ring of endomorphisms.
A ring $R$ is {\it stable} if every regular ideal of $R$ (that is, every ideal containing a nonzerodivisor) is stable\footnote{Our definition of a stable ring differs slightly from Sally and Vasconcelos \cite{SVBull,SV} in that we require only that every regular ideal is stable, not that every ideal is stable. However, for a reduced one-dimensional Cohen-Macaulay ring, the two definitions agree \cite[p.~260]{Rush2gen}.}.
Independently of Bass, Lipman \cite{Lipman}
 studied stable ideals in one-dimensional semilocal Cohen-Macaulay rings and showed how stability was reflected  in terms of invariants of the ring such as multiplicity, embedding dimension and the Hilbert function. (We recall the definitions of these terms in Section 2.1.) The terminology of ``stable'' ideal originates with Lipman; it reflects the stabilization of a certain chain of infinitely near local rings. Lipman in turn was motivated to introduce these 
 ideals as a way to unify ideas from Arf and Zariski involving singularities of plane curves,   and to produce the largest ring  between a one-dimensional local  Cohen-Macaulay ring $R$ and its integral closure having the same multiplicity sequence as $R$ \cite[Corollary 3.10]{Lipman}.

 Sally and Vasconcelos, motivated by the work of Bass and Lipman, studied  stable Noetherian  local rings in detail in \cite{SVBull, SV}, and showed that a reduced local Cohen-Macaulay ring with finite normalization  is  stable if and only if every ideal of $R$ can be generated by two elements, thus substantiating a conjecture of Bass and proving a partial converse to his theorem mentioned above. 
Reduced Noetherian local rings having finite normalization and every ideal generated by two elements 
 are known in the literature as {\it Bass rings} because of their importance in Bass's article \cite{Bass}. 
 Sally and Vasconcelos \cite[Example 5.4]{SV} gave the first example of a  Noetherian stable domain $R$  of multiplicity $>2$ (and hence without finite normalization). They used   a construction of Ferrand and Raynaud that exhibits $R$ as the preimage of a  derivation over a specific field of characteristic $2$.  
 Heinzer, Lantz and Shah \cite[(3.12)]{HLS} showed that this technique can be modified to produce over this same field of characteristic $2$ for each $e>2$ an analytically ramified Noetherian local  stable domain of multiplicity $e$. 
 More recently, analytically ramified  Noetherian stable local rings  have proved to be useful  theoretical tools  
 for classifying the rank one discrete valuation rings  (DVRs) that arise as the normalization of an analytically ramified  one-dimensional Noetherian local domain, as well as for describing properties of the generic formal fiber of a  Noetherian local domain \cite{GFF, OTrans}; see Remark~\ref{idealization remark} and Theorem~\ref{example list}.\footnote{The definitions of normalization, finite normalization and analytically ramified rings appear at the end of this section and in Section 2.1.}

 The main goal of this article is to tie all these results together.  The Noetherian assumption is not essential to most of our arguments, and thus we use general ideal-theoretic methods to work in a setting in which   
 $R$ is a one-dimensional local ring with regular maximal ideal.\footnote{We use the terminology of ``local ring'' for a ring with a unique maximal ideal. In particular, we do not require a local ring to be Noetherian.}   For example, 
we show in Theorem~\ref{big} that $R$ is  stable if and only if the integral closure $\overline{R}$ is a Dedekind ring with the property that every $R$-submodule containing $R$ is a ring. This generalizes a result of Rush \cite[Theorem 2.4]{RushRep} to rings that are not necessarily Noetherian.
A finer classification is possible by distinguishing when the integral closure of $R$ is a finitely generated $R$-module: 
 We show in Theorems~\ref{new Bass} and~\ref{new infinite} that $R$ is stable and has finite integral closure if and only if a suitable completion $\wt{R}$ of $R$ is a Bass ring, while $R$ is stable without finite integral closure if and only if there is a nonzero prime ideal $P$ of $\wt{R}$ such that $P^2 = 0$ and  $\wt{R}/P$ is a DVR. From this we deduce in Corollary~\ref{reduced cor} that $R$ is reduced and stable if and only if $R$ is a Bass ring or $R$ is a stable domain without finite integral closure. The rings in the first class, the Bass rings, are classical, while the second class, the stable domains without finite integral closure, have a rather transparent structure, as indicated by Theorem~\ref{GFF theorem}.  
 
 We also revisit the two-generator property in Section 5. For $R$ a  one-dimensional local ring with regular maximal ideal $M$, we show that $M^n$  is two-generated (i.e., can be generated by two elements), for some $n\geq 2$, if and only if every $M$-primary ideal of $R$ is two-generated. From this we deduce:
 
\medskip

{\noindent}{\bf Theorem 5.6.} {\it  Let $R$ be a   Cohen-Macaulay ring with maximal ideal $M$. Then these statements are equivalent.
 \begin{itemize}
 \item[{\em (1)}] $M^n$ is two-generated, for some $n \geq 2$.
 \item[{\em (2)}] $R$ has Krull dimension $1$ and  multiplicity at most $
 2$.
 \item[{\em (3)}] $R$ is one of the following:
 \begin{itemize}
 \item[{\em (3a)}] a Bass ring;
 \item[{\em (3b)}] a one-dimensional analytically ramified stable domain; or
 \item[{\em (3c)}] a ring containing a nonzero principal prime ideal $P$ such that $P^2 = 0$ and $R/P$ is a DVR.
 \end{itemize}
 \end{itemize}}
We also generalize a theorem of Greither to rings that are not {\it a priori} Noetherian: If $R$ is a local ring with regular maximal ideal and the  integral closure of $R$ is a Dedekind ring generated by two elements as an $R$-module, then $R$ is a Bass ring (Theorem~\ref{Greither}).

\smallskip

{\it Conventions.}  All rings are commutative with $1$.  The Jacobson radical of the ring $R$ is denoted $\Jac R$. The total ring of quotients of  $R$ is denoted $\QQ(R)$. 
The {\it integral closure of $R$} is the integral closure of $R$  in $\QQ(R)$ and is denoted $\overline{R}$. 
We say that $R$ has {\it finite integral closure} if $\overline{R}$ is a finitely generated $R$-module. 
When $R$ is reduced, $\overline{R}$ is the 
 {\it normalization} of $R$, and  $R$ has {\it finite normalization} if $\overline{R}$ is a finitely generated $R$-module.  
 We write $(R,M)$ for a local ring $R$ with maximal ideal $M$. 


 
 \section{Preliminaries and background}

\subsection{One-dimensional Cohen-Macaulay rings} 

We collect in this subsection properties of one-dimensional  Noetherian local rings that are needed in later sections.  Recall that a local ring $(R,M)$ is {\it analytically unramified} provided that the $M$-adic completion $\widehat{R}$ has no nontrivial nilpotent elements. 

\begin{proposition} \label{2.2} \label{2.5} {\em (\cite[Corollary 1.21]{GHR} and \cite[pp.~263-264]{Ma})} 
Let $(R,M)$ be a local ring with maximal ideal $M$ and integral closure $\overline{R}$. Then
\begin{itemize}
\item[{\em (1)}]  If $R$ is reduced and one-dimensional, and there exists $t>0$ such that $M^t$ is finitely generated, then $R$ is Noetherian. 

\item[{\em (2)}] If $R$ is Noetherian, $M$ is regular and $\overline{R}$ is finitely generated as an $R$-module, then $R$ is reduced. 

\item[{\em (3)}] If $R$ is Noetherian and analytically unramified, then $\overline{R}$ is finitely generated as an $R$-module. 
\end{itemize}
\end{proposition}


The {\it embedding dimension} of a  Noetherian local ring $R$ is the minimal number of elements needed to generate its maximal ideal $M$. For sufficiently large values of $n$, the length of $R/M^n$ is a polynomial in $n$ of degree $d = \dim R$.
The product of $d!$ and the  leading coefficient of this polynomial is the {\it multiplicity} of $R$.  See \cite{Ma} for more information on multiplicity.

\begin{proposition}  \label{2.3} \label{2.4} \label{2.5}  {\em (\cite[Propositions 11.1.10 and 11.2.1]{SH}, \cite[Theorem 1.1, p.~49]{Sbook} and \cite[Theorem 10.2, p.~90]{Matlis})} 
Let $(R,M)$ be a one-dimensional Cohen-Macaulay  ring with maximal ideal $M$ and integral closure $\overline{R}$. Then
\begin{itemize}
\item[{\em (1)}]  If $M^2 = mM$ for some $m \in M$, then the embedding dimension and multiplicity of $R$ agree. In particular, $R$ has minimal mulitplicity. 

\item[{\em (2)}] Every ideal of $R$ can be generated by $e$ or fewer elements, where $e$ is the multiplicity of $R$.  

\item[{\em (3)}] $R$ has finite integral closure    if and only if $R$ is analytically unramified. 

\end{itemize}
\end{proposition}

\subsection{Completions  of a one-dimensional local ring}

For a one-dimensional local ring $(R,M)$ that is not Noetherian, the $M$-adic completion $\widehat{R}$  of $R$ can be too coarse to contain useful information about $R$. For example, if $M = M^2$, then $\widehat{R} = R/M$. For this reason we work with another completion 
that agrees with the $M$-adic completion of $R$ if $R$ is one-dimensional and Noetherian.

\begin{notation} \label{completion notation}
{\em Let $R$ be a one-dimensional local ring with regular maximal ideal $M$, and let $m$ be a nonzerodivisor in $M$. 
We denote by $\wt{R}$ the $mR$-adic completion of $R$ and by $\widehat{R}$ the $M$-adic completion of $R$.  Since $R$ has Krull dimension $1$ and $m$ is a nonzerodivisor, $mR$ is $M$-primary, and hence
the ring $\wt{R}$ is independent of the choice of  nonzerodivisor $m \in M$. } 
\end{notation}

\begin{remark} \label{new remark} {\em Let $R$ be a one-dimensional local ring with regular maximal ideal $M$. 
If  (i) $M$ is finitely generated, or (ii) $M$ is stable, then $\widetilde{R}$ can be identified with $\widehat{R}$. In case (i)  some power of $M$ is contained in a principal regular ideal $mR$, and so the $mR$-adic and $M$-adic topologies agree on $R$. For case (ii), Lemma~\ref{primary stable}(6) implies that $M^2 = mM$ for some  $m \in M$. Thus $M^2 \subseteq mR$  and so again the two topologies coincide. 
}
\end{remark}

\begin{proposition}  \label{2.6} \label{2.7}  \label{2.8} {\em (\cite[Theorem 2.1(3), p.~11]{Matlis}, \cite[Theorem 2.1(4), p.~11]{Matlis} and \cite[Theorem 2.8(1), p.~19]{Matlis})}
Let $R$ be a one-dimensional  local ring   with regular maximal ideal $M$. 

\begin{itemize}
\item[{\em (1)}]  If  $I$ is a regular ideal of $R$, then $I\wt{R}$ is a regular ideal of $\wt{R}$. 

\item[{\em (2)}] 
 If $\lambda:R \rightarrow \wt{R}$ denotes the canonical mapping, then  $\wt{R} = \lambda(R) + x\wt{R}$ for all nonzerodivisors $x \in R$.

\item[{\em (3)}]  If $I$ is a regular ideal of $R$, then the mapping $\lambda$ in (2) induces an isomorphism $R/I \cong \wt{R}/I\wt{R}$ and $I = \lambda^{-1}(I\wt{R})$. 
  
\end{itemize}
\end{proposition} 

\begin{proof}  Matlis proves the statements in the proposition for  what he calls the ``R-completion'' $H$  of  a commutative ring $R$; that is, $H = \varprojlim R/I$, where $I$ ranges over the regular ideals of $R$. Since in our case $R$ is local of Krull dimension one and the maximal ideal $M$ of $R$ is regular, the $R$-completion of $R$ is simply $\wt{R}$.   Thus the cited references apply also to $\wt{R}$. 
\end{proof}

\subsection{Quadratic extensions}

Rush \cite[Proposition 2.1]{RushRep} has shown that, if $R$ is a stable ring,  every $R$-submodule of $\overline{R}$ that contains $R$ is a ring. In this section we review some properties of extensions that share this property.  

\begin{definition}  \label{quad def}
{\em
 An extension $R \subseteq S$ of rings is {\it quadratic} if  $xy \in xR + yR + R$ for all $x,y \in S$; equivalently,  every $R$-module between $R$ and $S$ is a ring.   
}
\end{definition}


\begin{remarks}  \label{quad remark} {\em 

(1) Quadratic  extensions were considered by Handelman \cite{Handelman} and Rush \cite{RushRep}, and more recently in \cite{GFF, OTrans, OPres}.

(2) 
If an extension $R \subseteq S$ of rings is quadratic, then, for every $x \in S$, we have $x^2 \in xR + R$; that is, every $x \in S$ is a root of a monic  polynomial  of degree at most $2$ with coefficients in $R$. Thus every quadratic extension is an integral extension.


(3) If $R \subseteq S \subseteq T$ is an extension of rings and $R \subseteq T$ is quadratic, then $R \subseteq S$ is quadratic. Also, if $I$ is an ideal of $S$ that is also an ideal of $R$, then $R/I \subseteq S/I$ is quadratic if and only if $R \subseteq S$ is quadratic.}
\end{remarks}

A key tool for analyzing  quadratic rings is  
 Handelman's  classification of finite-dimen\-sio\-nal algebras that are  quadratic extensions of a base field. 

\begin{lemma}
{\em (Handelman \cite[Lemma 5]{Handelman})} Let $F$ be a field and
let $S$ be a finite-dim\-en\-si\-onal  $F$-algebra such that $F
\subseteq S$ is a quadratic extension. Then $S$ is isomorphic
as an $F$-algebra to one of the following:  
\begin{itemize}
\item[{\em (i)}] $F$,

\item[{\em (ii)}]  a field extension of $F$ of degree $2$, 

\item[{\em  (iii)}]  
 a local ring with square zero maximal ideal and residue
field isomorphic to $F$,

\item[{\em (iv)}]  $F \times F$,  or 

\item[{\em (v)}]  $F \times F \times F$, in which case  $F = {\mathbb{F}}_2$. 
\end{itemize}\label{quadratic}
\end{lemma}

\begin{remark} \label{quadratic remark} {\em   If $K/F$ is a finite field extension that is a quadratic extension of rings in the sense of Definition~\ref{quad def}, then from Lemma~\ref{quadratic} it follows that  $[K:F] \leq  2$. Thus an extension of fields that is quadratic  in the sense of Definition~\ref{quad def} agrees with the usual notion of a quadratic field extension.}
\end{remark}


The following proposition is deduced in \cite{OFS} from Handelman's lemma. 

\begin{proposition} \label{at most 3} {\em \cite[Proposition 3.3]{OFS}} If $R \subseteq S$ is a quadratic extension, then there are at most three prime ideals of $S$ lying over any prime ideal of $R$. 
\end{proposition}

\subsection{Stable rings}

Noetherian  stable rings have been studied in a number of contexts; cf.~\cite{AHP, Bass, DK, Goe, Handelman,HLS,Lipman,Rush2gen, SVBull, SV,Zan2, Zan}. Also, the related class of archimedean stable domains have been investigated recently by Gabelli and Roitman \cite{GR} and shown to generalize features of one-dimensional stable domains, and hence Noetherian stable domains, to a wider setting. 

The following theorem  summarizes some of the results discussed in the introduction.

\begin{theorem} \label{Bass thm}  {\em  (cf.~\cite[Corollary 7.3]{Bass}, \cite{DK}, \cite[Theorem 1.4]{Rush2gen} and  \cite[Theorem 2.4]{SVBull}) } {Let $R$ be a reduced Cohen-Macaulay ring. 
If every ideal of $R$ can be generated by $2$ elements (equivalently, $R$ has Krull dimension $\leq 1$ and multiplicity $\leq 2$), then $R$ is a stable ring. Conversely, if 
 $R$ is a  reduced stable ring with finite normalization, then $R$ is a Bass ring.}
\end{theorem}

Thus in the case that  $R$ is an analytically unramified  local Cohen-Macaulay ring, $R$ is stable if and only if every ideal of $R$ can be generated by two elements.   The analytically ramified case is more complicated, as illustrated by results throughout the paper.  We single out in the following definition a specific  class of such rings. 

\begin{definition}  
{\em 
 A domain $R$ is a {\it bad stable domain} if $R$ is a one-dimensional  stable local domain without finite normalization.  
}
\end{definition}

Bad stable domains were studied in \cite{GFF,OTrans}. 
 The appellation ``bad'' here is borrowed from the title of Nagata's appendix in \cite{Na}, ``Examples of bad Noetherian rings,'' and refers to the fact that the domain does not have finite normalization.

 \begin{theorem} \label{GFF theorem} {\em \cite[Proposition 2.1, Theorems 3.4 and 4.2 and Corollaries 3.5 and~4.3]{GFF}}
 The following are equivalent for a local domain $R$ with quotient field $F$. 
 
 \begin{enumerate} 
 
 \item[{\em (1)}] $R$ is a bad stable domain. 
 
 \item[{\em (2)}] $\overline{R}$ is a DVR, $R \subseteq \overline{R}$ is a quadratic extension and $\overline{R}/R$ is a divisible $R$-module.
 
 \item[{\em (3)}] With $\wt{R}$ as in Notation~\ref{completion notation}, there is a nonzero prime ideal $P$ of $\wt{R}$ such that $P^2 = 0$ and $\wt{R}/P$ is a DVR.
 
 \item[{\em (4)}] $\overline{R}$ is a DVR and $\overline{R}/R \cong \bigoplus_{i \in I}F/\overline{R}$ for some index set $I$.
 
 \item[{\em (5)}] $\overline{R}$ is a DVR and $\overline{R}/R$ is a direct sum of divisible Artinian uniserial $R$-modules. 
 
 \end{enumerate}
 Moreover,  $R$ is a bad Noetherian stable domain of multiplicity  $e$ if and only if 
 the cardinality of the set $I$ in (4) is $e-1$. 
 \end{theorem}

Theorem~\ref{GFF theorem}(3) and a theorem of Lech's can  be used to guarantee the existence of  bad stable domains with prescribed completions:  

\begin{corollary} \label{Lech cor} 
The following are equivalent for  a  Noetherian local ring  $(R,M)$ that is complete in the $M$-adic topology. 
\begin{itemize}
\item[{\em (1)}] 
The ring  $R$ is the completion of a bad Noetherian stable domain.

\item[{\em (2)}] 
 No nonzero integer of $R$ is a zerodivisor and there is a  prime ideal $P$ of $R$ such that $P^2 =0$ and $R/P$ is a DVR.
 
 \end{itemize}  
\end{corollary} 

\begin{proof}  
That (1) implies (2) follows from Theorem~\ref{GFF theorem}(3).  Conversely, assume (2). 
By a theorem of Lech \cite[Theorem 1]{Lech}, the fact that $M \ne 0$, $M$ is not an associated prime of $R$, and no nonzero integer of $R$ is a zerodivisor implies that the ring $R$ is  the completion of a  Noetherian local domain $A$. By Theorem~\ref{GFF theorem}(3), $A$ is a bad  Noetherian stable domain. 
\end{proof} 

\begin{remark}  \label{idealization remark}
{\em  
The archetypal example of  a ring $R$ such as in Corollary~\ref{Lech cor} is produced by Nagata idealization: Let $V$ be a complete DVR, and let $L$ be a nonzero finitely generated free $V$-module. Define $V *L$ as a $V$-module to be $V \oplus L$, and view $V*L$ as a ring with  multiplication given by $(v_1,\ell_1)(v_2,\ell_2) = (v_1v_2,v_1\ell_2 + v_2 \ell_1)$ for all $v_1,v_2 \in V$ and $\ell_1,\ell_2 \in L$. With $P = 0 * L$, the ring $R= V * L$  is a complete local ring 
meeting the hypotheses of Corollary~\ref{Lech cor}, and hence $R$ is the completion of a bad Noetherian stable domain $A$. Moreover, since the multiplicity of $A$ is the same as that of its completion $R$, 
the multiplicity of $A$ is $e=1 + \rank \: L$.  
}\end{remark}
 
 
 To further motivate the class of bad stable domains, we mention three other contexts in which these rings arise. 
 An extension $U \subseteq V$ of DVRs is {\it immediate} if $U$ and $V$ have the same residue field and the maximal ideal of $U$ extends to the maximal ideal of $V$. The  ${\bf m}$-adic completion of a DVR is an immediate extension. 

\begin{theorem} \label{example list} {\em (\cite[Corollary 5.7 and Theorem 5.9]{GFF} and \cite[Theorem 4.1]{OTrans})}

\begin{itemize}

\item[{\em (1)}]  If $k$ is a field and $A$ is an affine $k$-domain with Krull dimension $d>1$ and quotient field $F$, then there exists a bad Noetherian stable domain $R$ between $A$ and $F$ having multiplicity $d$ and residue field finite over $k$. 

\item[{\em (2)}]  If $(A,M)$ is a  Noetherian local domain and  there exists a DVR  $V$ of $A$ that birationally dominates $A$ and satisfies $V = A + MV$, then the kernel of the exterior differential $d_{V/A}:V \rightarrow \Omega_{V/A}$ is a bad stable domain between $A$ and $V$.  

\item[{\em (3)}] If $U \subseteq V$ is an immediate extension of DVRs having quotient fields $Q$ and $F$, respectively, then  there is a one-to-one correspondence between bad  stable domains $R$ containing $U$ and having normalization $V$ and proper full $V$-submodules of the module  $\Omega_{F/Q}$ of K\"ahler differentials. 

\end{itemize}

\end{theorem}

\begin{remarks} {\em 
(1) 
 The   ring $R$ in Theorem~\ref{example list}(1) is constructed  along the following lines: For an appropriate local $A$-algebra $B$ contained in $F$ such that $B$ is  essentially of finite type over $k$, teh ring $R$ is defined to be $R = F \cap (\widehat{B}/P^{(2)})$, where $\widehat{B}$ is the ${\bf m}$-adic completion of $B$, $P$ is a certain prime ideal of $\widehat{B}$ and $P^{(2)}$ is the second symbolic power of $P$. That $R$ has multiplicity $d$   depends on the fact that $B$ is an excellent local ring \cite[Corollary 5.6]{GFF}.  
 
(2) 
Every bad Noetherian stable domain of positive characteristic must arise as in Theorem~\ref{example list}(2)   from a derivation \cite[Theorem 6.4]{OTrans}.

(3) 
 The one-to-one correspondence in Theorem~\ref{example list}(3)  
is used in \cite[Theorem 7.6]{OTrans}  to describe complete 
 DVRs that are the normalization of an analytically ramified  Noetherian local domain. 
For example, every equicharacteristic  complete DVR is the normalization of both Noetherian and non-Noetherian bad stable domains. }
\end{remarks}

%

\section{Quadratic rings}


In this section we introduce the notion of a quadratic ring and show that these rings exemplify  some key technical features of stable local rings.
 We require  first some preliminaries on stable ideals. 
  
  \begin{definitions} \label{easy def} {\em Let $R$ be a ring, and let $Q(R)$ be its total ring of quotients. 
  
  \begin{itemize}
  \item[{(1)}] 
An $R$-submodule $I$ of $\QQ(R)$ such that
$bI \subseteq R$ for some nonzerodivisor $b \in R$ is called a {\it fractional ideal} of $R$.  The fractional ideal $I$ is {\it regular} if it contains a nonzerodivisor. 

\item[(2)] 
 If  $I$ is a
{regular} fractional ideal of the ring $R$, then $I$ is  {\it invertible}  if
$I$ is a finitely generated fractional ideal such that $IR_M$ is a
principal fractional ideal of $R_M$ for each maximal ideal $M$ of
$R$. 

\item[(3)] The ring $R$ is {\it finitely stable}\footnote{The rings we call ``finitely stable'' were called ``stable'' by Rush in   \cite{Rush2gen,RushRep}.} if every finitely generated regular ideal of $R$ is stable. 
\end{itemize}}
\end{definitions}

\begin{remark} \label{inv remark} {\em Let $I$ be a regular fractional ideal of a ring $R$. Then  $I$ is {invertible} if and only if $IJ = R$ for some fractional ideal $J$ \cite[Proposition 2.3, p.~97]{KZ}.  If $R$ has only finitely many maximal ideals, then an invertible ideal is principal \cite[Exercise 2.11B, p.~33.]{Lam}.}
\end{remark}

\begin{notation} {\em  Let $I$ be a regular fractional ideal of a ring  $R$. Then, by \cite[(2.8)]{OFS}, 
  End$(I) = \{q \in \QQ(R):qI \subseteq I\}$. To simplify notation we write $E(I) = $ End$(I)$.}  
 \end{notation}

  We collect in the next lemma some basic properties of a regular ideal $I$ and its endomorphism ring $E(I)$.  
  
\begin{lemma}  With the  terminology of Definition~\ref{easy def}, 
 let  $I$ be a regular fractional ideal of a ring $R$. Then
 \label{primary stable}
 \begin{enumerate}
 
  \item[{\em (1)}]   $I$ is a regular fractional ideal of $\End(I)$.

 \item[{\em (2)}] 
 $I \subseteq {\rm Jac}(\End(\Jac R)))$.

 \item[{\em (3)}]  $I$ 
 is stable if and only if $I$ is an invertible fractional ideal of $E(I)$.  
 
 \item[{\em (4)}] 
 If $I^2 = AI$ for some invertible fractional ideal $A$ of $R$ with $A
\subseteq I$, then $I = A\End(I)$ and $I$ is stable.

\item[{\em (5)}]  If there exists $x \in I$ such that $I^2 = xI$, then $I$ is stable.  
 
 \item[{\em (6)}] Let  $R$ be 
finitely stable and local. Then $I$ is  stable  if
and only if $I^2 = xI$ for some $x \in I$.

 
 \end{enumerate}
\end{lemma}

\begin{proof}
(1)   This follows from the fact that every nonzerodivisor in $R$ is also a nonzerodivisor in $\Q(R)$.

(2)  
Let $x \in \Jac R$ and $w \in E(\Jac R)$. Then $wx \in \Jac R$, and so $1-wx$ is a unit of $R$, hence also  of $E(\Jac R)$. Since $w$ was arbitrary in $E(\Jac R)$, we have $x \in \Jac E(\Jac R)$.  


(3) This is proved in \cite[Proposition 2.11]{OFS}.
 
(4)  This is proved in \cite[Proposition 2.13]{OFS}.

(5) Since $I$ is regular and $I^2 \subseteq xR$, the ideal $xR$ is regular, hence invertible, and so 
this follows from (4). 

(6) Under the assumptions of (6), it is proved in \cite[Corollary 5.7]{OFS} that  a regular ideal $J$ of $R$ is stable if and only if $J$ is a principal ideal of $E(I)$. The characterization in (6) now follows from (5). 
\end{proof}






\begin{notation} \label{new notation}{\em 
  Let $R$ be a ring
such that $\Jac R$ is a regular ideal. We associate to $R$ a tower
of rings in $\Q(R)$ by defining
\begin{center} $R_0 := R, \: \: \: \: \: \: \: R_i  := \End(\Jac R_{i-1})$ for every $i \geq
1,  \: \: \: \: \: \: \: R_\infty := \bigcup_{i =1}^\infty R_i$.
\end{center}}
\end{notation}
Lemma~\ref{primary stable}(1) implies that for each $i$, $\Jac R_i$ is a
regular ideal of $R_{i+1} = \End(\Jac R_i)$. In Lipman's terminology \cite[Section 2]{Lipman}, if $R$ is a local Cohen-Macaulay ring and  the Jacobson radical of each $R_i$ is stable, then the ring $R_{i+1}$ is the blow-up of $R_i$ at its Jacobson radical and  the localizations of the rings $R_i$ at maximal ideals are the local rings ``infinitely near'' $R$.

\begin{definition} \label{quadratic ring def} {\em Let $R$ be a ring such that $\Jac R$ is a regular ideal. With Notation~\ref{new notation} we define  $R$ to be a {\it quadratic ring} if \begin{itemize} \item[(1)] $R$ is local, \item[(2)]  $R
\subseteq R_\infty$ is a quadratic extension, 
 and \item[(3)]  the maximal ideal
of $R$ is  stable and regular. 
\end{itemize}}
\end{definition}

As we show in Theorem~\ref{pre-connection}, the relevance of this notion is that  a finitely stable local ring with stable regular maximal ideal is a quadratic ring. 
Some technical properties of quadratic rings are elaborated in Lemmas~\ref{quadratic R1} and \ref{step one} and Proposition~\ref{properties}. 
Parts of  these results were proved for  finitely stable local domains with stable maximal ideal in \cite[Section 4]{OlbStructure}, but the arguments given here are somewhat different and more general in that they  permit zero divisors.

\begin{lemma} \label{quadratic R1}
 Let $R$ be a local ring with regular maximal ideal $M$ and residue field $F = R/M$. Let $R_1$ be  as in Notation~\ref{new notation}.
 Suppose that $R \subseteq R_1$ is a quadratic extension with
$R \ne R_1$.   
\begin{enumerate}

\item[{\em (1)}]  If $R_1$ is a local ring such that $\Jac R_1 = M$, then $R_1 = R_2$ and $R_1/M$ is a field extension of $F$ of degree $2$.

\item[{\em (2)}] If $R_1$ is a local ring such that $\Jac R_1 \ne M$, 
then $R_1 = R + M_1$ and $R_1/M_1 \cong F$. 

\item[{\em (3)}]
If  $R_1$ is not a
 local ring, then $\Jac R_1=M$, $R_1 = R_2$ and $R_1/M$
is isomorphic as an $F$-algebra to either
\begin{itemize} 
\item[{\em (i)}]  $F \times F$, or 
\item[{\em (ii)}] $F \times F \times F$, in which case $F = {\mathbb{F}}_2$.
\end{itemize}

\item[{\em (4)}]  $(\Jac R_1)^2 \subseteq M$.
\end{enumerate}

\end{lemma}

\begin{proof}
(1)  Suppose that $R_1$ is a local ring such that $ \Jac R_1=M$.  Then $R_1/M$ is a field and 
$\Jac R = M = \Jac R_1$. Hence   $R_1=R_2$. 
To see that  $R_1/M$ has degree $2$ over $F$, 
   let $x,y \in R_1 \setminus R$, and let
  $S = xR+ yR +R$.   
   Then $S$ is a ring since $R \subseteq R_1$ is a quadratic
  extension. Hence  
  $R/M \subseteq S/M$ is a quadratic extension, which by  
    Remark~\ref{quad remark} must be  integral.  Since $R_1/M$ is a field, the subring $S/M$ is an integral domain, and hence since $S/M$ is integral over the field $R/M$, $S/M$ is a field.  Since $R/M \subsetneq S/M$, 
  Lemma~\ref{quadratic}  implies that 
   $S/M$ has degree 2 as a field extension of  $R/M$.
 Since $x,y \not \in R$,
  this forces $S/M = (xR+R)/M = (yR+R)/M$.  Hence $x \in yR+R$
  for all $x,y \in R_1 \setminus R$, and so  for each
  $x \in R_1
  \setminus R$, we have $R_1 = xR+R$.  Therefore the degree of $R_1/M$ over
  $R/M$ is $2$. 

(2) Suppose that $R_1$ is a local ring such that $\Jac R_1 \ne M$.   Let $M_1$ denote the maximal ideal of $R_1$. We
claim
 that $R_1 = R + M_1$. By Lemma~\ref{primary stable}(2), $M \subsetneq M_1$. 
Let 
   $x \in M_1 \setminus M$, let $y \in
 R_1$, and let $S = R + xR + yR$.  
 Since  $x
 \in M_1 \cap S$, we have $M_1 \cap S \ne M$.
 Because $R \subseteq R_1$ is quadratic, $S$ is a ring, and,  by Remarks~\ref{quad remark}(2), $R \subseteq R_1$ is an integral extension. Since $R_1$ is local, so is $S$. 
   Also, $S/M$
cannot be a field since $M \subsetneq M_1 \cap S \subsetneq S$.  
Furthermore, since $S$ is local, the ring $S/M$ is not isomorphic to $F \times F$ or $F \times F \times F$. 
By Lemma~\ref{quadratic}, $S/M$ is a local ring with square
zero maximal ideal and residue field $F$. If $N$ is  the maximal ideal of $S$,
we have $N^2 \subseteq M$ and $S/N \cong F =  R/M$. Therefore  $y \in S =
R+N \subseteq R + M_1$. The choice  of $y \in R_1$ was arbitrary, and so $R_1 = R+M_1$, as desired. Moreover, since $M \subseteq M_1$, this implies that  $R_1$ has residue field $F$. 

(3) Suppose that $R_1$ is not a local ring, and let $k$
denote the number of maximal ideals of $R_1$.   
By Lemma~\ref{primary stable}(2), $M \subseteq \Jac R_1$, and so $R_1/M$ has $k>1$ maximal ideals. Also, since $R \subseteq R_1$ is a quadratic extension of rings,  
 $R/M \subseteq
R_1/M$ is a quadratic extension in which $R_1/M$ has more than one maximal ideal. By \cite[Lemma 5.2]{OFS},
 $R_1/M$ is isomorphic as an $F$-algebra to $
\prod_{i=1}^k F$, and $F = {\mathbb{F}}_2$ when $k = 3$. Also, since  we have obtained a decomposition of $R_1/M$ as a finite product
of fields, $\Jac R_1 = M$ and hence $R_1 =R_2$. 

(4) In light of (1) and (3), the  only case in which (4) is not immediate is that of (2). 
Suppose that   $R_1$ is local and $M \ne \Jac R_1$. By Lemma~\ref{primary stable}(2), $M \subsetneq \Jac R_1$. Let 
  $a,b \in \Jac R_1$ with $a,b \not \in  M$.
Define $T = aR +bR + R$. Then, as in the proof (2), Lemma~\ref{quadratic}
forces the square of the maximal ideal of the ring $T$ to be contained in
$M$. Thus $ab \in M$, and it follows that $(\Jac R_{1})^2 \subseteq M$.
 \end{proof}


For the next lemma, we recall that a ring $R$ is {\it Pr\"ufer} if every finitely generated regular ideal is invertible.

\begin{lemma} \label{Rush lemma} 
{\em (\cite[Proposition 2.1 and Theorem 2.3]{RushRep} and \cite[Corollary 5.11]{OFS})}
A  ring $R$ is finitely stable if and only if 
\begin{itemize}
\item[{\em (a)}] $R \subseteq \overline{R}$ is a quadratic extension, 
\item[{\em (b)}] $\overline{R}$ is a Pr\"ufer ring, and 
\item[{\em (c)}] there are at most two maximal ideals of $\overline{R}$ lying over each regular maximal ideal of $R$.
\end{itemize}
\end{lemma}

The next theorem shows that the class of quadratic rings contains the  finitely stable local rings with stable maximal ideal.  

\begin{theorem} \label{pre-connection} 
Let  $(R,M)$ be a  finitely stable local ring with stable regular maximal ideal $M$, and let  $R_\infty$ be defined as in Notation~\ref{new notation}. 
Then $R$ is a quadratic ring for which $R_\infty$ has at most two maximal ideals. 
\end{theorem}

\begin{proof} With $\{R_i\}_{i=0}^\infty$ as in Notation~\ref{new notation}, we prove the theorem by establishing four claims. 

\smallskip

  {\textsc{Claim 1}}. 
{\it  For each $i \geq 0$, if $R_i$ is a local ring with maximal ideal $M_i$, then $R_i \subseteq R_{i+1}$ is a quadratic extension.  }

\smallskip

We use the fact that $R_i \subseteq R_{i+1}$ is a quadratic extension if and only if every finitely generated $R_i$-submodule of $R_{i+1}$ containing $R_i$ is a stable fractional ideal of $R_i$; see \cite[Proposition 3.5]{OFS} or \cite[proof of Lemma 2.1]{Rush2gen}. Thus we let $A$ be a finitely generated $R_i$-submodule of $R_{i+1}$ with $R_i \subseteq A$. Let $m \in M$ be a nonzerodivisor. Then $m \in M_i$ and $$m \in mR_i \subseteq mA \subseteq mR_{i+1} \subseteq R_i,$$ since $R_{i+1} = \{q \in Q(R):qM_i \subseteq M_i\}$.  Therefore $mA$ is a finitely generated regular $R_i$-submodule of $R_i$, hence a finitely generated regular ideal of $R_i$. Every ring between a finitely stable ring and its total quotient ring is finitely stable \cite[Proposition 5.1]{OFS}. Hence $R_i$ is finitely stable, and so $mA$ is a stable ideal of $R_i$ by Definition~\ref{easy def}(3). 
Since $mA \cong A$, $A$ is a stable fractional ideal of $R_i$. This proves Claim 1. 

\smallskip

  {\textsc{Claim 2}}. 
{\it  For each $i \geq 0$ with $R_{i} \subsetneq R_{i+1}$, $R_i$ is local.} 

\smallskip

For $i = 0$, the claim holds, since $R = R_0$ is local. Suppose that $i>0$ and that $R_i$ is local whenever $R_i \subsetneq R_{i+1}$. If $R_{i+1} \subsetneq R_{i+2}$, then also $R_{i} \subsetneq R_{i+1}$ and so $R_i$ is local by induction. By Claim 1, $R_i \subsetneq R_{i+1}$ is a quadratic extension. Since $R_{i+1} \ne R_{i+2}$,  Lemma~\ref{quadratic R1}(3) implies that $R_{i+1}$ is  a local ring, as desired for Claim 2. 

\smallskip

  {\textsc{Claim 3}}. 
(i) {\it For each $i \geq 0$, $R_{i} \subseteq R_{i+1}$ is an integral extension.} (ii) {\it The extension $R \subseteq R_\infty$ is integral.} 

\smallskip

 The first statement follows from Claims 1 and 2 and Remark~\ref{quad remark}(2). The second statement follows from the first, since compositions of integral extensions are again integral \cite[Proposition 6,
   p.~307]{Bou}.

   \smallskip
   
   {\textsc{Claim 4}.}  
{\it  $R_\infty$ has at most two maximal ideals.}

\smallskip

  By Claim 3, $R_\infty \subseteq \overline{R}$, and by Lemma~\ref{Rush lemma}, $\overline{R}$ has at most two maximal ideals.  Since $\overline{R}$ is an integral extension of $R_\infty$, $R_\infty$ also has at most two maximal ideals. 

\smallskip

{\textsc{Claim 5}}. {\it $R$ is a quadratic ring.} 
   
\smallskip

By Lemma~\ref{Rush lemma}, $R \subseteq \overline{R}$ is a quadratic extension, and Claim 3 implies $R_\infty \subseteq \overline{R}$. By Remark~\ref{quad remark}(3), $R \subseteq R_\infty$ is a quadratic extension. By Definition~\ref{quadratic ring def}, $R$ is a quadratic ring. 
    \end{proof}

By adding to Lemma~\ref{quadratic R1} the assumption   that the
maximal ideal  of $R$ is stable, we obtain stronger results in the
case where $R_1$ is local.  In particular, $R_1$ inherits the
property of having a stable maximal ideal.

\begin{lemma}\label{step one} {
Let $R$ be a local ring with stable regular maximal ideal $M$. Suppose
that  $R \subseteq R_1$ is a quadratic
extension and $R_1$ is a local ring with maximal ideal $M_1$.

\begin{enumerate}

\item[{\em (1)}]  If $R_1 \ne R_2$, then $M_1 = MR_2$, $M_1$ is a stable ideal of $R_1$ and  $R_1 = R+M_1$.

\item[{\em (2)}]  If $R_1 = R_2$, then $M_1$ is a principal ideal of $R_1$ and $R_1 =
R +xR$ for every $x \in R_1 \setminus R$.

\end{enumerate}}
\end{lemma}

\begin{proof} By Lemma~\ref{primary stable}(3), $M$ is an invertible ideal of $R_1$. 
Since $R_1$ is a local ring, $M = mR_1$ for some nonzerodivisor $m \in M$. By Lemma~\ref{quadratic R1}(4), $M_1^2 \subseteq M = mR_1$, so that $M_1^2 = mA$ for some ideal $A$ of $R_1$.

\smallskip

{\textsc{Claim}}: (i) {\it If $A = R_1$, then $M_1$ is invertible in $R_1$ and $R_1 = R_2$.} (ii) {\it If $A \subsetneq R_1$, then $M_1^2 = mM_1$ and $M_1 = mR_2 = MR_2$.  } 

\smallskip

For (i), $A = R_1 \Longrightarrow mR_1 = M_1^2 \Longrightarrow M_1(m^{-1}M_1) = R_1 \Longrightarrow M_1$ is invertible. We have $$R_2 = \{q \in Q(R_1):qM_1 \subseteq M_1\} = \{q \in Q(R_1):qM_1M_1^{-1} \subseteq M_1M_1^{-1}\} = R_1,$$ and so (i) holds.

For (ii), $A \subsetneq R_1 \Longrightarrow M_1^2 = mA \subseteq mM_1 \subseteq M_1^2 \Longrightarrow M_1^2 = mM_1$. By Lemma~\ref{primary stable}(4) we have
$$M_1 = mR_2 \subseteq MR_2 \subseteq M_1R_2 \subseteq mR_2 \Longrightarrow M_1 = mR_2 = MR_2,$$ and so (ii) holds. 

\smallskip

We return to the proof of Lemma~\ref{step one}. For (1), $M_1 = MR_2$ and $M_1$ is invertible in $R_2$ by Claim (ii). Thus $M_1$ is stable by Lemma~\ref{primary stable}(3). Since $E(M) = R_1 \ne R_2 = E(M_1)$, we have $M \ne M_1$. By Lemma~\ref{quadratic R1}(2), $R_1 = R + M_1$. This verifies (1).

For (2), we have $R_1 = R_2$. If $M_1 = M$, we have $M_1 = mR_1$ is principal. By 
Lemma~\ref{quadratic R1}(1), the field $R_1/M$ has degree $2$ over $R/M$, and hence $R_1  = R + xR$ for all $x \in R_1 \setminus R$, and so (2) holds if $M_1 = M$.  

Thus we assume that $M_1 \ne M$.  We have $M_1 $ is a principal ideal of $R_1$ in either case $A = R_1$ or $A \subsetneq R_1$, using either (i) of the claim ($M_1$ is invertible and so is a principal ideal of $R_1$), or (ii) of the claim ($M_1 = mR_2 = mR_1$). Thus the first part of (2) holds. 

It remains to show $R_1 = R + xR$, for all $x \in R_1$, given that $R_1= R_2$, $M_1 = nR_1$, for some $n \in R_1$, and $M_1 \ne M$.  We have also $R \ne R_1$ and $n$ is  a nonzerodivisor. 

By 
Lemma~\ref{quadratic R1}(4), $M_1^2 \subseteq M$. If $mM_1 = M_1^2$, then $$mnR_1 = n^2R_1 \Longrightarrow M = mR_1 = nR_1 = M_1,$$a contradiction. Thus $mM_1 \subsetneq M_1^2 \subseteq M = mR_1$. Hence there exist $c,d \in M_1$ with $cd \in mR_1 \setminus mM_1$; say $cd = um$, where $u$ is a unit of $R_1$. Then $m = u^{-1}cd \in M_1^2$, and $$mR_1 \subseteq M_1^2 \subseteq M = mR_1 \Longrightarrow M_1^2 = M.$$ Therefore $$  M_1/M = M_1/M_1^2 \cong R_1/M_1 \cong R/M,$$ by Lemma~\ref{quadratic R1}(2). Since $R/M$ is a field, we have two equations. 
$$(3.10{\rm a}) \:\: M_1 = qR+M, \: \: \forall q \in M_1 \setminus M. \:\:\:\:\ \ (3.10{\rm b}) \:\: R_1 = R+M_1.$$ 

Let $x \in R_1 \setminus R$. By Equation 3.10b, $x = r + m_1$, for some $r \in R$ and $m_1 \in M_1$, where $m_1 \not \in M$, since $x \not \in R$. Then, using Equations 3.10b and 3.10a, we have $$R+xR \subseteq R_1 = R + M_1 \subseteq R+m_1R = R + (r+m_1)R = R + xR.$$
This completes the proof of the lemma. 
  \end{proof}




We collect now some of the main technical properties of quadratic rings that will be needed in the next section.

\begin{proposition}  \label{properties} Let $R$ be a quadratic
ring with  maximal ideal $M$, let $\{R_i\}_{i=0}^\infty$ and $\{R_\infty\}$ be as in Notation~\ref{new notation}, and let $$n = \min\{k \in \{0,1,2,\ldots,\infty\}: R_k = R_\infty\}.$$      Then the following statements hold for $R$.  

%

\begin{enumerate}

\item[{\em (1)}]  There exists a nonzerodivisor $m \in M$ such that $M = mR_1$.

\item[{\em (2)}]  For each $i < n$,  $R_i = R + M_i$ and $R_i$ is a quadratic ring with stable
maximal ideal $M_i = mR_{i+1}$. 


\item[{\em (3)}] If $n = \infty$, then $R_\infty$ is a local ring with maximal ideal $mR_\infty$ and $R_\infty = R + mR_\infty$. 

\item[{\em (4)}]   $R_\infty$ has at most $3$ maximal ideals, each of which is principal.

\item[{\em (5)}] If $R_\infty$ is not 
 local, then  $\Jac R_\infty =
mR_\infty$.


\item[{\em (6)}] The $R$-module $R_\infty$ is
 finitely generated  if and only if $n < \infty$. 

\item[{\em (7)}] Suppose  $n < \infty$ and 
 $J:=\bigcap_{i=1}^\infty M^i$ is a one-dimensional ideal of $R$. Then $J$ is   the intersection of the one-dimensional prime ideals of $R$, the number of which is not more than the number of maximal ideals of $R_\infty$.  Moreover $J$ is also the intersection of 
  the one-dimensional prime ideals of  $R_\infty$, and hence $J$ is an ideal of $R_\infty$. 




\end{enumerate}
\end{proposition}

\begin{proof}
(1) By Definition~\ref{quadratic ring def}, $M$ is stable. Since  $R \subseteq R_1$ is a quadratic extension,  Proposition~\ref{at most 3} implies that $R_1$ has at most 3 maximal ideals. By 
Lemma~\ref{primary stable}(3),
 $M$ is invertible in $R_1$. Since $R_1$ is local, $M = mR_1$, for some nonzerodivisor $m \in M$.

(2) 
The proof is by induction. If $n=0$ or $i=0$, the statement holds by (1). Suppose the result holds for some $i$ with $0 < i < n-1$. Then $R_i$ is a quadratic ring with stable maximal ideal $M_i = mR_{i+1}$ and $R_i = R + M_i$. Also $R \subsetneq R_{i} \subsetneq R_{i+1} \subsetneq R_{i+2}$, by the minimality of $n$. By Lemma~\ref{quadratic R1}(3), $R_{i+1}$ is a local ring. By Lemma~\ref{quadratic R1}(1), the maximal ideal $M_{i+1}$ of $R_{i+1}$ properly contains $M_i$. By Lemma~\ref{step one}(1), $M_{i+1} = M_iR_{i+2} = mR_{i+1}R_{i+2} = mR_{i+2},$ and $R_{i+1} = R_i + M_{i+1} = R + M_{i+1}$. Since 
 $R \subseteq R_\infty$ is a quadratic extension, it is clear that $R_{i+1} \subseteq R_\infty$ is also a quadratic extension, and so
 $R_{i+1}$ is a quadratic ring.

(3) If $n = \infty$, then by (2), each $R_i$ is local with maximal ideal $mR_{i+1}$ and $R_i = R + M_i$. Hence $R_\infty$ is a local ring with maximal ideal $mR_\infty$ and $R_\infty = R + mR_\infty$.

(4)  By Proposition~\ref{at most 3}, $R_n$ has at most three maximal ideals. If $n = \infty$, then (4) follows from (3). If $n = 0$, then $R = R_0$ is local and it has principal maximal ideal by (1), so that (4) holds.

Suppose that $0 < n < \infty$. Then $R_{n-1} \ne R_n$. By (2), $R_{n-1}$ is a quadratic ring and the maximal ideal $M_{n-1} = mR_n$ is stable. Also $R_n = R_{n+1}$. If $R_n $ is local, then Lemma~\ref{step one}(2) implies that the maximal ideal $M_n$ is principal, and so we have established (4). 

If $R_n$ is not local, then Lemma~\ref{quadratic R1}(3) implies that $\Jac R_n = M_{n-1} = mR_n$.  
Since $R_n$ has at most three maximal ideals, we may denote them 
 by $N_1,N_2,\ldots,N_t$, where $2 \leq t \leq 3$.  Then $mR_n = \bigcap_{i=1}^t N_i= \prod_{i=1}^t N_i$, by \cite[Theorem 1.3]{Ma}.  Then each $N_i$ is invertible; e.g., $N_1$ has inverse $m^{-1}\prod_{i=2}^t N_i$. By Remark~\ref{inv remark}, each of the $N_i$ is principal. Thus (4) holds.

(5) Since $R_\infty$ is not local, (3) implies that $n < \infty$. As noted in the proof of (4), we have $\Jac R_n = mR_n$.

(6)  If $n = \infty$, then clearly $R_\infty = \bigcup_{i=1}^\infty R_i$ is not finitely generated. 

For the converse, the case $n = 0$ is clear. Let $0 < n < \infty$. We show first that $R_n$ is a finitely generated $R_{n-1}$-module. 
\smallskip

Case 1: If $R_n$ is local, then Lemma~\ref{step one}(2) implies that $R_n = R_{n-1} + xR_{n-1}$, for every $x \in R_n \setminus R_{n-1}$, and so $R_n$ is finitely generated as an $R_{n-1}$-module if $R_n$ is local.

\smallskip

Case 2: If $R_n$ is not local, then, by  (3), $n < \infty$. By (2), $R_{n-1}$ is a quadratic ring, and so by Lemma~\ref{quadratic R1}(3), $R_n/M_{n-1}$ has dimension at most three as a vector space over $R_{n-1}/M_{n-1}$. Therefore $R_n$ is a finitely generated $R_{n-1}$-module. 

\smallskip

 If $n = 1$, then the proof is complete. 
If $n>1$, the result will follow from the following claim (using induction).  

\smallskip

{\textsc{Claim.}} {\it Suppose $0 \leq k \leq n-2$ and $R_{k+2}$ is finitely generated as an $R_{k+1}$-module. Then $R_{k+1}$ is finitely generated as an $R_k$-module. }

\smallskip

Write $R_{k+2} = b_1R_{k+1} + \cdots + b_tR_{k+1}$, for some $b_1, \ldots,b_t \in R_{k+2}$.  By (2), we have $M_{k+1} = mR_{k+2}$ and \begin{eqnarray*}
R_{k+1} & = & R_k + mR_{k+2}\:\:  = \:\: R_k + mb_1R_{k+1} + \cdots + mb_tR_{k+1} \\
\: & = & R_k + mb_1(R_k+mR_{k+2})+\cdots + mb_t(R_k+mR_{k+2}).
\end{eqnarray*} 
By Lemma~\ref{quadratic R1}(4), $m(mR_{k+2}) \subseteq M_{k+1}^2 \subseteq M_k$. Each $mb_i \in R_{k+1}$. Therefore $$R_{k+1} \subseteq R_k + mb_1R_k + \cdots + mb_tR_k + M_k \subseteq R_{k+1}.$$ This proves the claim.

(7) 
 By (4),
$R_\infty$ has at most $3$ maximal ideals, say
$N_1,\ldots,N_r$, all of which are regular and  principal. 
For each $j =1,\ldots,r$, since  $N_j$ is a principal regular ideal
of $R_\infty$, it follows  that $Q_j:=\bigcap_{k>0}N_j^k$ is a prime ideal and 
every prime ideal of $R_\infty$ properly contained in $N_j$ is contained in 
$Q_j$ \cite[Theorem 2.2]{AMN}.  Therefore each nonmaximal prime ideal of $R_\infty$ is contained in one of the $Q_j$. 
 Now $R \subseteq R_\infty$ is a quadratic, hence integral, extension. By assumption $J=\bigcap_{k>0}M^k $ is a one-dimensional ideal of $R$. Every prime ideal of $R_\infty$ containing $J$ also has dimension at most one. 
Since $J  \subseteq Q_1 \cap \cdots \cap Q_r$,   
each prime ideal $Q_i$ has dimension one. Since also each nonmaximal prime ideal of $R_\infty$ is contained in one of the $Q_i$, 
 the set of one-dimensional prime ideals of $R_\infty$ is   $\{Q_1,\ldots,Q_r\}$. Since $R_\infty$ is an integral extension of $R$, it follows that $\{Q_1 \cap R, \ldots, Q_r \cap R\}$  is the set of   one-dimensional prime ideals of $R$. 
%
 Now 
 since $\Jac R_\infty = N_1\cdots
N_r = N_1 \cap \cdots \cap N_r$, we have that $\bigcap_{k>0}(\Jac
R_\infty)^k = Q_1 \cap \cdots \cap Q_r$.
  In fact,  since $R_n = R_\infty$, an inductive argument using 
 Lemma~\ref{quadratic R1}(4) shows that 
 some power of the ideal  $\Jac
R_\infty$ is contained in $M$. Hence $J = \bigcap_{k>0}M^k =
\bigcap_{k>0}(\Jac R_\infty)^k = Q_1 \cap \cdots \cap Q_r$.  
 \end{proof}

\begin{corollary} \label{properties cor} Let $R$ be a one-dimensional quadratic ring with maximal ideal $M$,  let $J = \bigcap_{i=1}^\infty M^i$, let $\overline{R}$ be the integral closure of $R$, and let $\{R_i\}_{i=0}^\infty$ and $R_\infty$ be as in Notation~\ref{new notation}. Then 
\begin{itemize}
\item[{\em (1)}] $R_\infty = \overline{R}$. 

\item[{\em (2)}] 
The $R$-module $\overline{R}$ is finitely generated if and only if $R_\infty = R_n$ for some $n \geq 0$. 
 
 \item[{\em (3)}]  The ring $R/J$ has integral closure $R_\infty/J$ and  total quotient ring  $\Q(R)/J$.
 
 \end{itemize}

\end{corollary} 

\begin{proof} 
(1) 
Since $R \subseteq R_\infty$ is a quadratic, hence integral, extension, and $R$ has Krull dimension one,  $R_\infty$ also has  Krull dimension one. 
Since there are only finitely many maximal ideals of the one-{dimension\-al} ring $R_\infty$, all of which are principal and regular, it follows that every regular ideal of $R_\infty$ is principal, and hence the integral extension $R_\infty$ of $R$ is integrally closed in $\Q(R)$; cf.~\cite[Theorem 10.18, p.~237]{LM}.

(2)  This follows from (1) and   Proposition~\ref{properties}(6). 

(3)  By Proposition~\ref{properties}(1), there exists a nonzerodivisor $m \in M$ such that $M^2 = mM$.   Since $R$ has dimension one,  $\Q(R) = R[1/m]$. Also, since $M^2 = mM$, it follows that $J = (R:_R R[1/m])$, and so $J$ is an ideal of $\Q(R)$.  Let $m^*$ denote the image of $m$ in $R/J$. Then $m^*$ is a nonzerodivisor in the one-dimensional ring $R/J$. Hence $Q(R)/J = R[1/m]/J = (R/J)[1/m^*] = \Q(R/J)$. By (1), $\overline{R} = R_\infty$, and so it follows that  $R_\infty/J$ is the integral closure of $R/J$ in $Q(R/J) = Q(R)/J$. 
%
\end{proof}






By Proposition~\ref{properties}(4), $R_\infty$ has at most three maximal ideals if $R$ is a quadratic ring. 
In the next theorem we describe 
the quadratic rings $R$ for which $R_\infty$ has exactly three maximal ideals.

\begin{theorem} \label{quad char} Let  $R$ be a local ring with a regular maximal ideal. Then the following statements are equivalent. 

\begin{itemize}

\item[{\em (1)}] $R$ 
is a quadratic ring such that ${R}_\infty$ has three maximal ideals. 

\item[{\em (2)}] There is a ring $S$ between $R$ and $\Q(R)$ such that 
\begin{itemize}

\item[{\em (i)}] $S$  
has exactly 3 maximal ideals, 

\item[{\em (ii)}]  each maximal ideal of $S$ 
is principal and  regular with residue field isomorphic to
${\mathbb{F}}_2$, and 

\item[{\em (iii)}]  $R = k + \Jac S$, where $k$ is the prime subring of $R$. 

\end{itemize}
\end{itemize}
\end{theorem}

\begin{proof} (1) $\Rightarrow $ (2)   Suppose that $R$ is a quadratic ring and $R_\infty$ has three maximal ideals.  Since $R$ is local, $R \ne R_\infty$, and hence $R \ne R_1$.  
If $R_1 = R_\infty$, then,  
by Lemma~\ref{quadratic R1}(3), 
$M = \Jac R_1$ and the three maximal ideals of  $R_1$ each have residue field isomorphic to ${\mathbb{F}}_2$. Moreover, 
 by Proposition~\ref{properties}(4),  
each maximal ideal of $R_1 $ is regular and principal. 
Since $(k + M)/M \subseteq R/M$ and $R/M$ has only $2$ elements (because $R/M$ is a subfield of $R_1/M \cong  {\mathbb{F}}_2 \times {\mathbb{F}}_2 \times {\mathbb{F}}_2$), it must be that $k + M = R$.  
Therefore to verify (2) it remains to show that  $R_1 = R_\infty$ and choose $S = R_1$.

 Since $R$ is a quadratic
ring, each $R_i$ with $R_i \ne R_{i+1}$ is 
a quadratic ring by Proposition~\ref{properties}(2). Denote the maximal ideal of each such
$R_i$ by $M_i$.
 Since $R_\infty$ has $3$ maximal ideals, Proposition~\ref{properties}(3) implies that there is
 $n>0$ such that $R_{n-1} \subsetneq R_n = R_{\infty}$. We show that $n = 1$. 
Suppose by way of contradiction that $n>1$. 
%
By Proposition~\ref{properties}(2),
 $R_{n-2}$ is a quadratic ring with $R_{n-2} \subsetneq R_{n-1} \subsetneq R_{n} = R_\infty$, and, by Lemma~\ref{quadratic R1}(2), 
  $R_n/M_{n-1} \cong {\mathbb{F}}_2 \times {\mathbb{F}}_2 \times
{\mathbb{F}}_2$.   To show that such a case is impossible, we may after relabeling assume  that $n = 2$. We show the assumption that 
$R$ is a quadratic ring with $R \subsetneq R_1 \subsetneq R_2=R_\infty$ and the fact that $R_2/M_1 \cong {\mathbb{F}}_2 \times {\mathbb{F}}_2 \times
{\mathbb{F}}_2$ lead to a contradiction of the fact that $R \ne R_1$. 

\smallskip

{\textsc{Claim}}: {\it  The  dimension of the $R/M$-vector space $R_1/R$
is $1$.}  

\smallskip  Let $A$ be an $R$-module such that $R \subseteq A
\subsetneq R_1$. It suffices to show that $R = A$.
  Since $R \subseteq R_1$ is a quadratic
extension, $R \subseteq R_1$ is an integral extension and $A$ is a ring. Thus $A \subseteq R_1$ is an
integral extension. Since $R_1$ is local,  $A$ is
local with maximal ideal, say $N$, lying over $M$ and under
$M_1$.   By Proposition~\ref{properties}(1) and (2), there is a nonzerodivisor $m \in M$ such that  $M = mR_1$ and $M_1 = mR_2$.    Then $mR_1 = M \subseteq N \subseteq M_1 =
mR_2$, so that $R_1 \subseteq m^{-1}N \subseteq R_2$.  Since $R_1
\subseteq R_2$ is a quadratic extension, $m^{-1}N$ is a ring, and
hence $N = mA_1$, where $A_1 = \End(N) = m^{-1}N$. Moreover $R_1 \subseteq A_1
\subseteq R_2$.  To prove that $R = A$, we show that $R_1 = A_1$. Once this is established, we have that $A$ has maximal ideal $N = mA_1 = mR_1 =
M$. Since by Lemma~\ref{step one}(1), $R$ and $R_1$ have the same residue field $R/M$, so does $A$. Thus the  fact that $N = M$ 
 implies  that $A = R + N = R$,  as claimed. 

For the claim it remains to show that with $A$ and $A_1$ as above, $R_1 = A_1$.  Since $R_1/M_1 \subseteq A_1/M_1 \subseteq R_2/M_1 \cong {\mathbb{F}}_2 \times {\mathbb{F}}_2 \times
{\mathbb{F}}_2$, 
 and the last ring is reduced,  $A_1/M_1$ is also reduced. Hence $M_1 = \Jac A_1$. 
Since  
 every element of the ring
$ {\mathbb{F}}_2 \times {\mathbb{F}}_2 \times
{\mathbb{F}}_2$ is an idempotent, 
either $R_1 = A_1$ or 
 $A_1/M_1$ is a decomposable ring.  
 Suppose that $R_1 \ne A_1$. Then  $A_1/M_1$ is  a decomposable ring, and hence $A_1 $ 
has more than one maximal ideal. Since $R \subseteq R_\infty$ is a quadratic extension and $R \subseteq A\subseteq A_1 \subseteq R_\infty$, we have that $A \subseteq A_1$ is also a quadratic extension with $A_1$  not local.
By 
Lemma~\ref{quadratic R1}(3)  (applied to $A \subseteq A_1$), $N =
\Jac A_1.$ We have established above that $\Jac A_1 = M_1$, and so $N =  M_1$. Therefore, by 
Lemma~\ref{quadratic R1}(1), $R_1 = R + M_1 = R+ N \subseteq A$,
contrary to the choice of $A$ with $A \subsetneq R_1$. 
 This shows $A_1/M_1$ cannot be a decomposable ring, which forces $R_1 = A_1$ and  proves that $R_1/R$ has dimension $1$ as an $R/M$-vector space. This verifies the claim.

\smallskip

We show now that the claim implies a contradiction to our assumption that  $R \ne R_1$. Let $t \in R_1 \setminus R$. 
Since $R_1/R$ has dimension $1$, 
we have $R_1 = R + tR$, and hence $M = mR_1 = mR + mtR$, so that as an $R/M$-vector space,
$M/M^2$ has dimension at most $2$.  By
Lemma~\ref{quadratic R1}(4), $M_1^2 \subseteq M$. Also, $M_1^2/M^2$
is a proper subspace of $M/M^2$, and as such has dimension at most
$1$.  (We have $M_1^2 \subsetneq M$  since $M_1 =
mR_2$, $M= mR_1$  and $m$ is a nonzerodivisor.)  However, $M_1^2/M^2 = m^2R_2/m^2R_1
\cong R_2/R_1$, so that $R_2/R_1$  has dimension $1$ as an
$R/M$-vector space. There is an exact sequence of $R/M$-vector
spaces,
$$0 \rightarrow R_1/M_1 \rightarrow R_2/M_1 \rightarrow R_2/R_1
\rightarrow 0.$$ The left and right  vector spaces have
dimension $1$, while the middle vector space has dimension $3$,
since $R_2/M_1 \cong {\mathbb{F}}_2 \times {\mathbb{F}}_2 \times {\mathbb{F}}_2$.  This contradiction
implies that the assumption $R \subsetneq R_1 \subsetneq R_2$ is impossible under the conditions in (1). Therefore $R_1 = R_2$, from which it follows that $R_1 = R_\infty$. This completes the proof that (1) implies (2). 

(2) $\Rightarrow$ (1) 
  By (2),  $M=\Jac S$ is the product of the three
principal maximal ideals of $S$, so  $\Jac S$ is a principal regular ideal
of $S$. Hence $\Jac S$ is a stable maximal ideal of $R$.
Moreover $S = \End(M) = R_1 = R_\infty$ and $R_1/M \cong {\mathbb{F}}_2 \times {\mathbb{F}}_2 \times {\mathbb{F}}_2$. 
By Lemma~\ref{quadratic}, $R/M \subseteq R_1/M$ is  a quadratic extension, and hence $R \subseteq R_1=R_\infty$ is also a quadratic extension by Remark~\ref{quad remark}. Therefore  $R$ is a quadratic ring and $R_\infty$ has three maximal ideals. 
 \end{proof}

\section{Characterizations of one-dimensional stable rings}

In this section we work in  the following setting.

\begin{setting}\label{tilde}  {\em Let 
 $R$ be a one-dimensional local ring 
 having a regular maximal ideal
$M$, let $\overline{R}$ denote the  integral closure of $R$ in its total quotient ring $\QQ(R)$, and let  $m \in M$ be a nonzerodivisor. Let $J = \bigcap_{i>0}m^iR$. As in Notation~\ref{completion notation}, we denote by $\wt{R}$ and $\widehat{R}$ the completions of $R$ in the
$mR$-adic and $M$-adic topologies, respectively.  Both $J$ and $\wt{R}$ are independent of the choice of nonzerodivisor $m$ in $M$. Moreover, since $m$ is a nonzerdivisor,  $R/J$ is a one-dimensional ring with regular maximal ideal $M/J$.} 
\end{setting}

We first characterize  a one-dimensional  stable local ring in terms of its  integral closure.  
Recall that a ring $A$ is a {\it Dedekind ring} if every regular ideal $I$ of $A$ is invertible. 

\begin{theorem} \label{big} 
With Setting~\ref{tilde}, the following statements are equivalent.

\begin{enumerate}

\item[{\em (1)}] $R$ is a stable ring.

\item[{\em (2)}]  $R$ is a finitely stable ring with stable maximal ideal.



\item[{\em (3)}] $R$ is a quadratic ring such that
$\overline{R}$ has at most two maximal ideals.

\item[{\em (4)}] $R \subseteq \overline{R}$ is a quadratic 
extension and $\overline{R}$ is a {Dedekind  ring} with  at most two maximal ideals.

  \item[{\em (5)}] 
    $R/J$ is a stable ring.
    

\end{enumerate}


  \end{theorem}

  \begin{proof}  
(1) $\Rightarrow$ (2) This is clear.

(2) $\Rightarrow$ (3) This follows from  Lemma~\ref{Rush lemma} and Theorem~\ref{pre-connection}.

(3) $\Rightarrow$ (4)  By Corollary~\ref{properties cor}(1),
$\overline{R} = R_\infty$, and so $R \subseteq \overline{R}$ is a
quadratic  extension. By (3), $\overline{R}$ has at most 2
maximal ideals, and Proposition~\ref{properties}(4) implies the maximal ideals of
$\overline{R}$ are principal regular ideals. Since $\overline{R}$ has dimension one,  every regular ideal of $\overline{R}$ is a principal regular ideal, and hence $\overline{R}$ is a Dedekind ring.


(4) $\Rightarrow$ (1)
Let $I$ be a regular ideal of $R$.  Since $\overline{R}$ has Krull
 dimension $1$ and at most 2 maximal ideals, both of which are
 principal and regular (they are regular because they lie over $M$), it follows that
every regular ideal of $\overline{R}$ contains a power of  $\Jac
\overline{R}$.  
Since $\overline{R}$ is a Dedekind ring with only finitely many maximal ideals, every regular ideal is principal. In particular, $I\overline{R}$ is  a principal regular ideal; that is,
$I\overline{R} = x\overline{R}$ for some nonzerodivisor $x \in
\overline{R}$.
  An argument of Rush  \cite[proof of Theorem 2.2]{RushRep}, which is also given explicitly in \cite[Proposition 3.6]{OFS}, shows that since $\overline{R}$ has at most two maximal ideals, $x$    can be chosen  an element of
$I$. 
 Therefore  $R \subseteq x^{-1}I \subseteq \overline{R}$.
Since $R \subseteq \overline{R}$ is a quadratic extension, $x^{-1}I$
is a ring, and hence $I = x\End(I)$.  This shows that
every regular ideal of $R$ is stable.   

(1) $\Rightarrow$ (5) Let  $I$ be an ideal of $R$ such that $I/J$ is a regular ideal in $R/J$.  Then $I$ is an $M$-primary ideal $R$ and hence  a regular ideal in $R$. Since $R$ is stable, Lemma~\ref{primary stable}(6) implies that $I^2 = xI$, for some $x \in I$. 
Since $I^2$ is $M$-primary, we have  $J \subseteq I^2 \subseteq xR$, 
 and hence  $I^2/J = (I/J)(xR/J)$. Moreover, 
since $I^2 \subseteq xR$, the ideal $xR$ is $M$-primary, hence regular in $R$. Therefore  
 $xR/J$ is $M/J$-primary in $R/J$ and hence regular in $R/J$. Lemma~\ref{primary stable}(4) implies then that $I/J$ is a stable ideal in $R/J$, which proves that $R/J$ is a stable ring. 
 
(5) $\Rightarrow$ (1) Let $I$ be a regular ideal of $R$. 
Then $I$ is $M$-primary, and so there is $k>0$ such that $m^k \in I$. Since the image of $m^k$ in $R/J$ is a nonzerodivisor, $I/J$ is regular. By (5) and Lemma~\ref{primary stable}(6) there exists $x \in I$ such that $I^2/J = (xI + J)/J$.  Thus $I^2 = xI +J$, so that since $J$ is an ideal of $R[1/m]$, we have  
 $R[1/m] = I^2R[1/m] = xIR[1/m] + JR[1/m] = xR[1/m] + J$. Therefore there exist $\ell \geq 0$, $r \in R$ and $j \in J$ such that $1 = x(r/m^\ell) + j$. Thus $(1-j)m^\ell \in xR$, and, since $1-j$ is 
   a unit in $R$, $m^{\ell} \in xR$. Therefore $J \subseteq m^{k+\ell}R \subseteq xI$, so that $I^2 = xI + J = xI$. By Lemma~\ref{primary stable}(5), $I$ is stable. 
\end{proof}

\begin{remark} {\em Rush \cite[Theorem 2.4]{Rush2gen} has proved the equivalence of (1) and (4) of Theorem~\ref{big} in the case where  $R$ is a Noetherian ring. Our proof that (4) implies (1) follows Rush's argument.} \end{remark}

From the theorem we deduce a characterization of one-dimensional quadratic rings that shows  these are stable except in a very special case. 

 \begin{corollary} A one-dimensional local ring  $R$ is a   quadratic ring if and only if either $R$ is a stable ring or $R =   k + \Jac S$,  where $k$ is the prime subring of $R$ and $S$ is a Dedekind ring between $R$ and $\Q(R)$ with exactly three maximal ideals, 
  each with residue field isomorphic to  
${\mathbb{F}}_2$.
 
 \end{corollary}

 \begin{proof} Suppose that $R$ is a quadratic ring that is not a stable ring. 
   By 
      Theorem~\ref{big}, $\overline{R}$ has more than two maximal ideals. By Corollary~\ref{properties cor}(1), $\overline{R} = R_\infty$, and so, by Proposition~\ref{properties}(4), $R_\infty$ has exactly three maximal ideals. 
      By Theorem~\ref{quad char}, 
 there exists a ring $S$ between $R$ and $\Q(R)$    having exactly 3 maximal ideals,
all of which are principal and have residue field isomorphic to
${\mathbb{F}}_2$, such that $R= k + \Jac S$. 
As noted in the proof  of Theorem~\ref{quad char}, we may take $S$ to be $R_\infty$. 
Since the maximal ideals of the one-dimensional ring $\overline{R}=R_\infty$ are principal and regular,  $\overline{R}$ is a Dedekind ring. 
The converse is a consequence of Theorems~\ref{pre-connection} and~\ref{quad char}.  
 \end{proof}






In Theorem~\ref{new Bass}  we characterize the one-dimensional  stable local rings with finite integral closure; Theorem~\ref{new infinite} similarly characterizes  stable rings without finite integral closure.  
\begin{theorem} \label{new Bass} 
With Setting~\ref{tilde}, the following statements are equivalent. 
\begin{enumerate}
\item[{\em (1)}] $R$ is a  stable ring with finite integral closure.

\item[{\em (2)}]  $R/J$ is a Bass ring.

\item[{\em (3)}]  $\wt{R}$ is a Bass ring.
 
 \end{enumerate}
\end{theorem}

\begin{proof} 
(1) $\Rightarrow$ (2) 
By Theorem~\ref{big}, $R/J$ is a stable ring.
 Since also $R$ has finite integral closure, 
 Corollary~\ref{properties cor}(2) implies that
  $R_n = R_\infty$ for some $n$. By Proposition~\ref{properties}(7), $J$ is the nilradical of $R$. Therefore $R/J$ is a reduced one-dimensional stable ring. 
%
  To show that $R/J$ is a Bass ring, it remains   by Theorem~\ref{Bass thm} 
 to show that  $R/J$ is a Noetherian ring with finite integral closure.
 By Proposition~\ref{properties}(7), since $R_n = R_\infty$, $J$ is the intersection of the minimal prime ideals $Q_1,\ldots,Q_r$  of $R_\infty$, and, as noted in the proof of the proposition, these prime ideals are comaximal.  
    Now $R_\infty$ is a one-dimensional ring with principal maximal ideals (Proposition~\ref{properties}(4)), and so for each $i$, $R_\infty/Q_i$ 
is a principal ideal domain.  Since  $Q_1,\ldots,Q_r$ are pairwise comaximal, the ring   $R_\infty/J$ is the product of the principal ideal domains $R_\infty/Q_i$.  Therefore 
 $R_\infty/J$ is a
principal ideal ring.
By Proposition~\ref{properties}(6), 
$R_\infty/J$ is a finitely generated  $R/J$-module, and so 
 the Eakin-Nagata
Theorem \cite[Theorem 3.7, p.~18]{Ma} implies  that $R/J$ is a Noetherian ring. Moreover, by Corollary~\ref{properties cor}(3), $R_\infty/J$ is the integral closure of $R/J$ in its total quotient ring, and so $R/J$ has finite integral closure. Therefore $R/J$ is a Bass ring. 

 (2) $\Rightarrow$ (3)  Let $M$ denote the maximal ideal of $R$. 
Since $S:=R/J$ is reduced and has finite normalization, Proposition~\ref{2.5}(3) implies that 
 the $M/J$-adic completion 
$\widehat{S}$ of $S$ is also reduced, and hence, again by Proposition~\ref{2.5}(3), $\widehat{S}$  has finite normalization. Since $S$ has multiplicity $\leq 2$, so does $\widehat{S}$. Thus $\widehat{S}$ is a Bass ring. To show then that $\wt{R}$ is a Bass ring, we need only prove that $\wt{R} \cong \widehat{S}$.  
 Since $mR/J$ contains a power of  $M/J$,  the $(mR/J)$-adic completion of $S= R/J$ coincides with the $(M/J)$-adic completion $\widehat{S}$ of $S$. 
For each $i >0$,  $J \subseteq m^iR$,  so that   $S/m^iS \cong R/m^iR$ as $R$-algebras. Since these isomorphisms are natural,  $$\widehat{S} = \varprojlim S/m^iS \cong \varprojlim R/m^iR = \wt{R},$$ 
which completes the proof of (3).



(3) $\Rightarrow$ (2)  As noted in the proof of (2) implies (3), the $mR/J$-adic completion of $R/J$ coincides with $\wt{R}$, so we may assume without loss of generality that $J = 0$ and show that $R$ is  a Bass ring. 
Since $J = 0$, we may identify  $R$ with its image in $\wt{R}$. Also, since  
 $\wt{R}$ is a reduced ring,  so is $R$.

We claim  that $R$ is Noetherian. Since $R$ is reduced and  has Krull dimension $1$, it suffices by Proposition~\ref{2.2}(1) to show that $M$ is finitely generated. 
%
To this end, observe that since $\wt{R}$ is a Noetherian ring,  the $\wt{R}$-module $M\wt{R}/m\wt{R}$ is finitely generated. 
 By Proposition~\ref{2.7}(2), $\wt{R} = R + m\wt{R}$, and so $M\wt{R}/m\wt{R}$ is finitely generated  as an $R$-module. By Proposition~\ref{2.8}(3),  
$M/mR \cong M\wt{R}/m\wt{R}$ as $R$-modules,  and hence $M/mR$ is a finitely generated $R$-module. We conclude that $M$ is a finitely generated ideal of $R$ and hence that $R$ is  a Noetherian ring.

Since $R$ is Noetherian,  the $M$-adic and $mR$-adic topologies agree on $R$. Therefore the $M$-adic completion $\widehat{R}$ of the  Noetherian ring $R$ is a Bass ring. In particular,   $R$ is an analytically unramified  Noetherian local ring.  Since $\widehat{R}$ has multiplicity $\leq 2$, so does $R$, which proves that $R$ is a Bass ring. 

(2) $\Rightarrow$ (1) By Theorem~\ref{big}, $R$ is  stable ring, and by Corollary~\ref{properties cor}(3), $R_\infty/J$ is the integral closure of $R/J$. Since $R/J$ is a Bass ring, $R_\infty/J$ is a finitely generated $R$-module. Hence $R_\infty$ is a finitely generated $R$-module, so that   $R$ has finite integral closure  by Corollary~\ref{properties cor}(1).  
\end{proof}

A simple example shows that the conditions of the theorem are not in general equivalent to the condition that $R$ is a Bass ring. 
The example, as well as later examples in this section, are constructed  using Nagata idealization, which was discussed in Remark~\ref{idealization remark}. 
%

\begin{example} {\em Let $V$ be a DVR with quotient field $F$. Then   $R= V \star F$ is a one-dimensional  stable local ring with regular maximal ideal \cite[Lemma 3.3]{GFF}. The ring $R$ has 
 finite integral closure (it is in fact integrally closed), and $R$ is not a Bass ring because it is not reduced, nor even Noetherian.} 
\end{example}

We turn next to the characterization of stable rings without finite integral closure.

\begin{theorem} \label{new infinite} 
With Setting~\ref{tilde},
the following statements are equivalent. 

\begin{enumerate} 

\item[{\em (1)}] $R$ is a stable ring without finite integral closure.

\item[{\em (2)}] $R/J$ is a stable ring without finite integral closure.

\item[{\em (3)}] There exists a nonzero prime ideal $P$ of $\wt{R}$ such that $P^2 = 0$ and $\wt{R}/P$ is a DVR.
\end{enumerate}

\end{theorem} 

\begin{proof}  
(1) $\Rightarrow$ (2)  Suppose that $R$ is a stable ring without finite integral closure. By Theorem~\ref{big}, $R/J$ is a stable ring. By Corollary~\ref{properties cor}(1) and (3), $\overline{R}/J$ is the integral closure of $R/J$. 
Thus, since  $R$ does not have  finite integral closure, neither does $R/J$.  


(2) $\Rightarrow$ (3)  As observed in the proof of (2) implies  (3) of 
Theorem~\ref{new Bass}, 
 the $(mR/J)$-adic completion of $R/J$ coincides with $\wt{R}$. We assume without loss of generality that  $J = 0$ and $R$ is a stable ring without finite integral closure. Corollary~\ref{properties cor}(2) implies that $R_i \ne R_\infty$, for all $i >0$. By Proposition~\ref{properties}(3) there exists a nonzerodivisor $x \in M$ such that $xR_\infty$ is the maximal ideal of $R_\infty$ and $R_\infty = R + xR_\infty$.  To simplify  notation, let $S = R_\infty$ and $Q = Q(R)$.

\smallskip

{\textsc{Claim 1}:}  {\it The inclusion mapping $R \rightarrow S$ lifts to a surjective ring homomorphism $\wt{R} \rightarrow \wt{S}$, where $\wt{S}$ denotes the completion of $S$ in the $mR$-adic completion of $S$.  }

\smallskip
   
  Consider the exact sequence $$\Hom_R(Q/R,Q/R) \stackrel{\delta}{\rightarrow}  \Hom_R(Q/R,Q/S) \rightarrow \Ext_R^1(Q/R,S/R).$$  
  Since $x$ is a nonzerodivisor  in $R$, the ideal $xR$ is $M$-primary. Since $R$ has Krull dimension one, this implies that some power of $x$ is contained in every principal regular ideal of $R$.  Consequently, since 
   $S = R+ xS$,  it follows that $S = R + rS$ for all nonzerodivisors $r \in R$; that is, $S/R$ is divisible with respect to the set of nonzerodivisors of $R$. This property, along with the fact that $Q=R[1/m]$ is a countably generated $R$-module, implies that $\Ext_R^1(Q/R,S/R) = 0$ \cite[Theorem 1.1]{AHT}, and hence the map $\delta$ is a surjection.  
 By \cite[Theorem 2.2, p.~13]{Matlis},  $\Hom_R(Q/R,Q/R)$ can be identified with $\wt{R}$, while $\Hom_R(Q/R,Q/S)$ can be identified with $\wt{S}$.  Since these identifications are natural,  $\wt{R}$ maps onto $\wt{S}$. As discussed in the proof of \cite[Theorem 2.9, p.~21]{Matlis}, this mapping is a ring homomorphism.  
 
 \smallskip
 
 {\textsc{Claim 2}:}  {\it There is a prime ideal $P$ of $\wt{R}$ such that $\wt{R}/P$ is a DVR.} 
 
 \smallskip
 
Since $x$ is a nonzerodivisor in $S$, Proposition~\ref{2.6}(1)  implies that $x$ is also a nonzerodivisor in $\wt{S}$. 
Since $S$ has principal maximal 
ideal $xS$,  $\wt{S}$ has  principal maximal ideal 
 $x\wt{S}$. Thus, since $x$ is a nonzerodivisor in $\wt{S}$, the ideal $L = \bigcap_{k}x^k\wt{S}$ has residue ring $\wt{S}/L$ that is  a DVR. Since Claim 1 implies that $\wt{R}$ maps onto $\wt{S}/L$, there is a prime ideal $P$ of $\wt{R}$ such that $\wt{R}/P$ is a DVR. 
 \smallskip

 {\textsc{Claim 3}:} {\it The ideal $P$ in Claim 2 has the property that $P^2 = 0$.}  
 
 \smallskip
 
 The ideal $P$ in Claim 2  
  is the kernel of the induced map $\wt{R} \rightarrow \wt{S}/\bigcap_{k}x^k\wt{S}$. With $m$ as in Setting~\ref{tilde}, 
since $mR$ is $M$-primary, $\bigcap_{k}x^k\wt{S} = \bigcap_{k} m^k\wt{S}$. Viewing $\wt{R}$ and $\wt{S}$ as subrings of $\prod_{i >0} R/m^iR$ and $\prod_{i>0} S/m^iS$, respectively, we have 
$$P = \{\l r_i + m^iR\r \in \wt{R}: \l r_i + m^iS \r \in \bigcap_{k} m^k\wt{S}\}.$$  
Let $p,q \in P$, and write $p = \l a_i + m^iR \r$ and $q = \l b_i + m^iR \r$.  Since $p,q \in P$, there exist $c_i,d_i, s_i,t_i \in S$ such that  $a_i = m^{i+1}c_i + m^is_i$ and $b_i = m^{i+1}d_i + m^it_i$.  Thus $a_i, b_i \in m^iS \cap R$. 
Since $S = R+mS$, $R \subseteq S \subseteq R[1/m]$ and $R \subseteq S$ is a quadratic extension, it follows that 
 $(m^iS \cap
R)^2 \subseteq m^iR$ for all $i>0$ \cite[Lemma 3.2]{GFF}.
Therefore $a_ib_i \in m^iR$ for all $i>0$, which shows that $pq = 0$, and hence $P^2 = 0$.  

(3) $\Rightarrow$ (1)  Since $\wt{R}/P$ is a DVR and $P^2 = 0$,  $\wt{R}$ is a one-dimensional stable ring \cite[Lemma 3.3]{GFF}. 
We claim that this implies  $R$ is a stable ring. To this end,
 we first make an observation about the canonical homomorphism   $\lambda:R \rightarrow \wt{R}$. 
 
 \smallskip
 
 {\textsc{Claim:}} {\it 
 If $z \in \wt{R}$ is such that
 $\lambda^{-1}(z\wt{R})$ is a regular ideal of $R$, then 
 the principal ideal $z\wt{R}$ is generated by the image of an element of $\lambda^{-1}(z\wt{R})$.}
 
 \smallskip  
 

Let $z \in \wt{R}$ be such that $\lambda^{-1}(z\wt{R})$ is a regular ideal of $R$, and let 
 $x$ be a nonzerodivisor in
 $\lambda^{-1}(z\wt{R})$. Write $\lambda(x) = zw$ for some $w \in
 \wt{R}$.  If $w$ is a unit in $\wt{R}$, then $x\wt{R} = z\wt{R}$,
 as claimed.  Suppose that $w$ is not a unit in $\wt{R}$.
   Since $x$ is a nonzerodivisor in $R$, Proposition~\ref{2.7}(2) gives that   $\wt{R} = \lambda(R) + x\wt{R}$. 
  Therefore $z \in \lambda(R) +
 x\wt{R} = \lambda(R) +  zw\wt{R}$, and so, for some $t \in
 \wt{R}$ and $y \in R$, $z(1-wt) = z - zwt =
 \lambda(y)$.  Since $w$ is not a unit and  $\wt{R}$ is
 a local ring,  $1 -wt$ is a unit in $\wt{R}$.
 Since $z(1-wt) = \lambda(y)$, we have $z\wt{R} = y\wt{R}$. This shows  that in all cases  $z\wt{R}$ is generated as a principal ideal by the image of an element of $R$. This proves the claim.

\smallskip

We use the claim now to  verify that $R$ is stable. Let
 $I$ be a regular ideal of $R$.
 Since  $\wt{R}$ is a  stable local
ring,  Lemma~\ref{primary stable}(6) implies that there exists $z \in
I\wt{R}$ such that $I^2\wt{R} = zI\wt{R}$. Since $I^2\wt{R}
\subseteq z\wt{R}$, we have $I^2 \subseteq \lambda^{-1}(z\wt{R})$.  Since $I^2$ is a regular
ideal of $R$,  the claim  implies that there
exists $x \in \lambda^{-1}(z\wt{R})$
 such that $z\wt{R} = x\wt{R}$. By Proposition~\ref{2.8}(3), $
 \lambda^{-1}(I\wt{R})  = I$, and so $x \in \lambda^{-1}(z\wt{R}) \subseteq \lambda^{-1}(I\wt{R}) =  I$. 
Furthermore, 
 $I^2\wt{R} = xI\wt{R}$, so again by Proposition~\ref{2.8}(3), $I^2 = \lambda^{-1}(I^2\wt{R}) = \lambda^{-1}(xI\wt{R}) =  xI$.  By Lemma~\ref{primary stable}(4),  $I$ is a stable ideal of $R$, which proves that $R$ is  stable. 
Finally,  if $R$ has finite integral closure, then, by Theorem~\ref{new Bass}, $\wt{R}$ is a Bass ring. In particular $\wt{R}$ is reduced, contrary to (3). Thus $R$ does not have finite integral closure.  
\end{proof}





In the case where $R$ is separated in the $mR$-adic topology (e.g., if $R$ is Noetherian), we obtain the following classification of one-dimensional  stable local rings. 

\begin{corollary} \label{previous}
With Setting~\ref{tilde}
 and $J = 0$, the ring $R$ is  stable  if and only if $R$ is a Bass ring or there is a nonzero prime ideal $P$ of $\wt{R}$ such that $\wt{R}/P$ is a DVR and $P^2 = 0$. 
\end{corollary} 

\begin{proof}
Apply Theorems~\ref{new Bass} and~\ref{new infinite}.
\end{proof}

In particular, if $R$ is complete in the $mR$-adic topology, then there are only two classes of (one-dimensional) stable rings: Bass rings and rings having a square zero prime ideal whose residue ring is a DVR. Thus,  in the case in which $R$ is $mR$-adically complete, $R$ is stable and reduced if and only if $R$ is a Bass ring.  More generally, we have the following corollary.

\begin{corollary} \label{reduced cor} With Setting~\ref{tilde}, the ring $R$ is a reduced  stable ring if and only if $R$ is either a Bass ring or a bad stable domain.
\end{corollary} 

\begin{proof}
Suppose that $R$ is a reduced stable ring.  If $R$ has finite normalization, then by Corollary~\ref{properties cor}(2), $R_i = R_\infty$, for some $i$, and so, by Proposition~\ref{properties}(7), $J = 0$. In this case  Theorem~\ref{new Bass} implies that $R$ is a Bass ring. Otherwise, suppose that $R$ does not have finite normalization.  By Corollary~\ref{properties cor}(2), $R_i \ne R_\infty$, for all $i>0$, and so 
 Proposition~\ref{properties}(3) implies $R_\infty$ is local. In this case   Proposition~\ref{properties}(7) implies that $R$ has a  
 unique minimal prime ideal, and hence, since $R$ is reduced, $R$ is a domain. By Theorem~\ref{new infinite} there exists a nonzero prime ideal $P$ of $\wt{R}$ such that $\wt{R}/P$ is a DVR and $P^2 = 0$. 
By Theorem~\ref{GFF theorem},  $R$ is a bad stable domain. The converse is clear in light of Theorem~\ref{Bass thm}.
\end{proof}

The next two examples illustrate some complications with characterizing the non-reduced, non-complete one-dimensional stable local rings $R$  such that $\wt{R}$ has a prime ideal $P$ with $P^2 = 0$ and $\wt{R}/P$ is a DVR. Such a ring $R$ has a nonzero prime ideal $Q$ such that $Q^2 =0$.
The first example shows that  even if $R$ is Noetherian, $R/Q$ need not be a DVR. The second example shows that the maximal ideal of $R$ can be two-generated but the ring $R$ not be Noetherian. 

\begin{example} {\em {\it A  local Cohen-Macaulay stable ring $R$ having a nonzero prime ideal $P$ such that $P^2 =0$ and $R/P$ is a bad stable domain (and hence $R/P$ is not a DVR).} 
Let $(A,M)$ be a bad Noetherian stable domain,  let $R = A \star (A/M)$ be the Nagata idealization of the $A$-module $A/M$ (cf.~Remark~\ref{idealization remark}) and let $P = 0 \star (A/M)$. Then $P^2 = 0$ and $R/P \cong A$ is a bad stable domain. 
Since $A/M$ is a finitely generated $A$-module, $R$ is a one-dimensional local Cohen-Macaulay ring. We claim that $R$ is stable. Let $m$ be a nonzerodivisor in $M$. 
  Then $\wt{R} = \wt{A} \star (\wt{A}/M\wt{A})$, where $\wt{R}$ is the $mR$-adic completion of $R$ and $\wt{A}$ is the $mA$-adic completion of $A$. 
   Since $A$ is a bad stable domain,  Theorem~\ref{GFF theorem} implies there is a nonzero prime ideal $Q$ of $\wt{A}$ such that $\wt{A}/Q$ is a DVR and $Q^2 = 0$. Now $L:= Q \star (\wt{A}/M\wt{A})$ is a prime ideal of $\wt{R}$, and $L^2 = 0$ since $Q^2 = 0$ and $Q\cdot (\wt{A}/M\wt{A}) = 0$. Also, $\wt{R}/L \cong \wt{A}/Q$, so that $\wt{R}/L$ is a DVR. By  Corollary~\ref{previous}, $R$ is a   stable ring. 
   If $A$ is chosen instead  to be a non-Noetherian bad stable domain, then $R$ is a non-Noetherian stable ring.    
}
\end{example}

\begin{example} \label{power example} {\em 
{\it A non-Noetherian  separated stable local ring $(R,M)$ of Krull dimension one whose maximal ideal can be generated by two elements.}
Let $k$ be a field,  let $X$ be an indeterminate over $k$, and let $V = k[X]_{(X)}$. The $(X)$-adic completion of $V$ is $\widetilde{V} = k[[X]]$. Define $R  = V \star \widetilde{V}$. Since $\widetilde{V}$ is not a finitely generated $V$-module, $R$ is not a Noetherian ring. 
However, the maximal ideal $M = X V \star \widetilde{V}$ of $R$ is generated by two elements,  $(X,0)$ and $(0,1)$, the first a nonzerodivisor in $R$ and the second a zerodivisor.  Since the ring $R$ has a prime ideal whose square is zero and whose residue ring is a DVR,  $R$ is a stable ring \cite[Lemma 3.3]{GFF}. Moreover  $\bigcap_{i}M^i = \bigcap_{i} (m^iV \star m^i\wt{V}) = 0$, and so $R$ is separated in the   $M$-adic topology.  
 }
\end{example}

\section{The two-generator property}

In this section we consider the two-generator property in our context.  
By a {\it proper power} of the ideal $I$, we mean an ideal of the form $I^k$, where $k>1$.  
Sally \cite[Proposition 1]{Sally} has shown that if a proper power of a regular ideal (which, {\it a priori}, is not necessarily finitely generated)  is two-generated, then so is every power of the ideal, including the ideal itself. 
We use this in the next proposition to connect the property of being two-generated with that of being stable.  

\begin{proposition} \label{Sally} Let $I$ be a regular ideal of the local ring $R$. Then, for some $n>1$, $I^n$ is  two-generated if and only if $I$ is  two-generated  with $I^2 = aI$ for some $a \in I$. 
\end{proposition}

\begin{proof}  Suppose that a proper power of $I$ is a two-generated ideal. By the result of Sally discussed above,   $I$ is a two-generated ideal and we may write $I = (x,y)R$ and $I^2 = (x^2,xy,y^2)R$. The argument that follows is based on the proof of \cite[Theorem 3.4]{SV}.  Since $I^2$ is two-generated,  Nakayama's lemma implies that  $I^2 = (x^2,xy)R = xI$, $I^2 = (xy,y^2)R=yI$ or $I^2 = (x^2,y^2)R$. In the first two cases, we have that $I^2 = aI$ for some $a \in I$. 
 Consider the last case, $I^2 = (x^2,y^2)R$. 
Then  $xy = rx^2+sy^2$ for some $r,s \in R$. If $r$ is a unit, then $x^2 \in (xy,y^2)R = yI$, so that $I^2 = yI$. Similarly, if $s$ is a unit, then $I^2 = xI$. If neither $r$ nor $s$ is a unit, then $xy \in MI^2$, with $M$ the maximal ideal of $R$.
Thus $I^2 = (x^2,y^2)R = (x^2-xy,xy-y^2)R + MI^2$, and by  
  Nakayama's lemma, $I^2 = (x^2-xy,xy-y^2)R = (x-y)I$.  This proves that in all cases $I^2 = aI$ for some $a \in R$. The converse of the proposition is clear.    
\end{proof}

\begin{corollary}  Let $R$ be a reduced one-dimensional local ring with regular maximal ideal $M$. If, for some $n>1$, $M^n$ is two-generated,  then  
every ideal of $R$ is two-generated.
\end{corollary}

\begin{proof} Suppose $M^n$ is two-generated for some $n>1$. By the result of Sally \cite[Proposition 1]{Sally} discussed above, $M$ is two-generated. 
Also, by Proposition~\ref{Sally}, $M^2 = mM$ for some $m \in M$.  Since $R$ is reduced and one-dimensional, Proposition~\ref{2.2}(1) implies that $R$ is a Noetherian ring.  Also,  since $M^2 = mM$ and $M$ is two-generated, Proposition~\ref{2.3}(1) implies that $R$ has multiplicity at most $2$. By Proposition~\ref{2.4}(2)  every ideal of $R$ is two-generated.   
\end{proof}

Rush \cite[Proposition 2.5]{RushRep} has shown that if $R$ is a local ring for which every finitely generated regular ideal is  two-generated, then $R$ is finitely stable. Proposition~\ref{Sally} leads to a different proof of this result in    slightly stronger form.

\begin{corollary} \label{Rush cor} If every finitely generated regular ideal of the local ring $R$ has a power that is two-generated, then $R$ is a finitely stable ring. 
\end{corollary} 

\begin{proof}
Let $I$ be a finitely generated regular ideal of $R$. By assumption $I^2$ has a power that is two-generated. Hence $I$ has a proper power that is two-generated.  By Proposition~\ref{Sally}, $I^2 = aI$ for some $a \in I$, and so $I$ is stable  by Lemma~\ref{primary stable}(5).  
\end{proof}

\begin{theorem} \label{power thm} 
With Setting~\ref{tilde},  
the following statements are equivalent. 
\begin{itemize}

\item[{\em (1)}]  $M^n$ is two-generated, for some $n \geq 2$. 

\item[{\em (2)}] $\wt{R}$ is a   Noetherian local ring of multiplicity at most $2$.

\item[{\em (3)}] Every regular ideal of $R$ is two-generated.


\end{itemize}

\end{theorem}

\begin{proof} 
(1) $\Rightarrow$ (2) By Proposition~\ref{Sally}, $M$ is two-generated and $M^2 = mM$ for some $m \in M$.  Thus the maximal ideal $M\wt{R}$ of $\wt{R}$ is two-generated and $M^2\wt{R} = mM^2\wt{R}$.  Since $M^2 = mM$, Remark~\ref{new remark} implies $\wt{R} = \widehat{R}$. Since $\bigcap_{i}M^i\widehat{R} = 0$ and  $\widehat{R}$ is $M\widehat{R}$-adically complete with finitely generated maximal ideal, the ring $\widehat{R}$ is Noetherian \cite[Theorem 3]{Cohen}. 
By Proposition~\ref{2.6}(1), $m$ is a nonzerodivisor in $\widehat{R}$  with $M^2\widehat{R} = mM\widehat{R}$,  and so by Proposition~\ref{2.3}(1) the multiplicity of $\widehat{R} = \wt{R}$ is 
 at most $2$.

(2) $\Rightarrow$ (3) Let $I$ be a regular ideal of $R$. By Proposition~\ref{2.6}(1), $I\wt{R}$ is a regular ideal of $\wt{R}$. Since $\wt{R}$ has Krull dimension 1 and multiplicity at most $2$, Proposition~\ref{2.4}(2) implies that every ideal of $\wt{R}$ is two-generated.  Since $I$ is regular, 
 $m^k \in  I$ for some $k>0$, and so  $I\wt{R}/m^{k+1}\wt{R} \cong I/m^{k+1}R$ by Proposition~\ref{2.8}(3).   By Proposition~\ref{2.7}(2), $\wt{R} =\lambda(R) + m^{k+1}\wt{R}$, where $\lambda:R \rightarrow \wt{R}$ is the canonical embedding. Since $I\wt{R}/m^{k+1}\wt{R}$ can be generated by two elements as an $\wt{R}$-module,  $I/m^{k+1}R$ can be generated by two elements as an $R$-module. Therefore $I = (x,y,m^{k+1})R$ for some $x,y \in I$. Since $m^{k+1} \in MI$, Nakayama's lemma implies that $I$ can be generated by two elements. 
 
 (3) $\Rightarrow$ (1) This is clear. 
\end{proof}

\begin{remark} {\em 
Example~\ref{power example} is a non-Noetherian ring satisfying the conditions of Theorem~\ref{power thm}. 
}\end{remark}

Restricting to Noetherian rings, we have a stronger characterization.

\begin{theorem} \label{power cor}  Let $R$ be a   Cohen-Macaulay ring with maximal ideal $M$. Then these statements are equivalent.
 \begin{itemize}
 \item[{\em (1)}] $M^n$ is two-generated, for some $n \geq 2$.
 \item[{\em (2)}] $R$ has Krull dimension $1$ and  multiplicity at most  $
 2$.
 \item[{\em (3)}] $R$ is one of the following:
 \begin{itemize}
 \item[{\em (3a)}] a Bass ring;
 \item[{\em (3b)}] a bad stable domain; or
 \item[{\em (3c)}] a ring containing a nonzero principal prime ideal $P$ such that $P^2 = 0$ and $R/P$ is a DVR.
 \end{itemize}
 \end{itemize}
\end{theorem}

\begin{proof} (1) $\Rightarrow$ (2) 
 Proposition~\ref{Sally} implies that every power of $M$ is two-generated, and hence   the Hilbert polynomial of $M$ is constant. Since the degree of this polynomial is one less than the Krull dimension of $R$, the Hilbert-Samuel theorem  \cite[p.~4]{Sbook} implies that  $R$ has Krull dimension one. 
By  
 Theorem~\ref{power thm}, $\wt{R}$ is a stable ring  of multiplicity at most $2$. Since $R$ is Noetherian and has Krull dimension one, Remark~\ref{new remark} implies   $\wt{R} = \widehat{R}$. Therefore, since $\widehat{R}$ has multiplicity $2$, so does $R$.

 (2) $\Rightarrow$ (3) By Proposition~\ref{2.4}(2) every ideal of $R$ is two-generated, and hence $R$ is a stable ring by Corollary~\ref{Rush cor}. If $\overline{R}$ is a finitely generated $R$-module, then $R$ is a Bass ring by Theorem~\ref{new Bass} and $R$ satisfies (3a). 
 
 Suppose  that $\overline{R}$ is not a finitely generated $R$-module. Since $R$ is one-dim\-en\-sion\-al and Noetherian, Remark~\ref{new remark} implies $\widehat{R}  = \wt{R}$. By Theorem~\ref{new infinite} there exists a nonzero prime ideal $L$ of  $\widehat{R}$   such that $L^2 = 0$ and $\widehat{R}/L$ is a DVR.    Viewing $R$ as a subring of $\widehat{R}$, we consider two cases. 
 
 \smallskip
 
{\bf  Case 1}: $L \cap R = 0$. 

\smallskip 

In this case 
 $R$ is a domain. By Theorem~\ref{GFF theorem}, $R$ is a bad stable domain and hence (3b) holds for $R$.  
 
 \smallskip
 
{\bf  Case 2:}  $P := L \cap R$ is a nonzero prime ideal of $R$.
 
 \smallskip
 
 We show that (3c) holds. 
 Since $L^2 =0$, we have  that $P^2 =0$.  
  Since $R$ is  stable, Lemma~\ref{primary stable}(6) implies there exists $m \in M$ such that $M^2 = mM$. We claim that $P \not \subseteq mR$.  Suppose  to the contrary that 
$P \subseteq mR$. Since $P$ is a prime ideal and $m \not \in P$, it follows that $P = mP$. Hence $P \subseteq \bigcap_i m^iR = \bigcap_i M^i= 0$, a contradiction. Thus $P \not \subseteq mR$.

Next we claim that $M^2 \subsetneq mR$. Indeed, if $M^2 = mR$, then, since $m$ is a nonzerodivisor, $M$ is an invertible, hence principal, ideal of $R$. Thus $R$ is a regular local ring, hence a domain. However, $P^2 = 0$ and $P$ is nonzero, so this contradiction shows that $M^2 \subsetneq mR$.

 Now since $M^2 \subsetneq mR$, $P \not \subseteq mR$ and  $M/M^2$ is a vector space of dimension at most $2$ over $R/M$, it follows that $M = mR + P$. Thus $R/P$ is a DVR. 
 Since $M = mR + P$ and $P^2  =0$, we have 
  $P/MP = P/(mP + P^2) = P/mP = P/(mR \cap P) \cong (P+mR)/mR = M/mR$. Since $M/M^2$ has dimension at most $2$ as a vector space over $R/M$ and $M^2 \subsetneq mR$, it follows that 
  $P/MP$ has dimension  at most one as an $R/M$-vector space. By Nakayama's lemma,  $P$ is a principal ideal, and so (3c) is satisfied by $R$.  

(3) $\Rightarrow$ (1) It is clear that if $R$ is a Bass ring, then every power of $M$ is two-generated. Similarly, if $R$ is a bad stable domain with multiplicity $\leq 2$, then, by Proposition~\ref{2.4}(2), every power of $M$ is two-generated. Finally, suppose there is a nonzero principal prime ideal $P$ of $R$ with $R/P$ a DVR and $P^2 = 0$. Since $R/P$ is a DVR and $P^2 =0$, it follows that $M^2 = mM$ for some $m \in M$ 
 \cite[Lemma 3.3]{GFF}.  Since $P$ is a principal ideal and $R/P$ is a DVR, $M$ is two-generated. Thus $M^2=mM$ is also two-generated, 
 which verifies (1).  
\end{proof}

Greither \cite[Theorem 2.1]{Gre} has shown using multiplicity theory  that  a reduced   Noetherian local ring  $R$ is a Bass ring if and only if  $\overline{R}$ can be generated by two elements as an $R$-module. 
Theorem~\ref{Greither}, which uses a different method of proof, generalizes this result to rings that {\it a priori} need not be Noetherian. 




\begin{lemma}\label{is quad} If $R \subseteq S$ is an extension of  rings such that $S$ can be generated by two elements as an $R$-module, then $R \subseteq S$ is a quadratic extension.
\end{lemma}

\begin{proof} By \cite[Lemma 3.1]{OFS}, it suffices to show that $R_M \subseteq S_M$ is a quadratic extension for each maximal ideal $M$ of $R$. Thus we may assume without loss of generality that $R$ is  a local ring. 
Since $S$ can be generated by two elements, Nakayama's lemma implies that $S = R + xR$ for some $x \in S$. Let $s,t \in S$. We claim that 
 $st \in sR + tR +R$.  Write $s = a + xb$ and $t = c + xd$ for some $a,b,c,d \in R$.  Since $S = R + xR$ and $S$  is a ring,  $x^2 \in R+xR$. Hence  
\begin{eqnarray*} st & = & ac + bcx + adx  + bdx^2 \:
  \in \: R  + bxR + dxR + dx^2R \\
 & \subseteq &  R + bxR + dxR \: = \: R + sR + tR, \end{eqnarray*}
This proves $R \subseteq S$ is a quadratic extension.
\end{proof}


Recall that a {\it generalized local ring} is a local ring with finitely generated maximal ideal $M$ such that $\bigcap_{i} M^i  = 0$; cf. \cite{Cohen}. 

\begin{theorem} \label{Greither}
The following are equivalent for a one-dimensional local ring $R$ with regular maximal ideal $M$. 

\begin{itemize} 

\item[{\em (1)}] $R$ is a Bass ring.

\item[{\em (2)}] $R$ is a  
generalized local ring whose integral closure  $\overline{R}$ can  be generated by two elements as an $R$-module.  

\item[{\em (3)}] $\overline{R}$ is a Noetherian ring that can be generated by two elements as an $R$-module. 
\end{itemize}
\end{theorem} 

\begin{proof} (1) $\Rightarrow$ (2) Suppose $R$ is a Bass ring. Then $R$ is a  generalized local ring. Since  $\overline{R}$ is a finitely generated $R$-submodule of $\Q(R)$,  there is a nonzerodivisor $r \in R$ such that $r\overline{R}$ is an ideal of $R$. Since $R$ is a Bass ring, $r\overline{R}$ is a two-generated ideal of $R$. Since $r$ is a nonzerodivisor, $\overline{R}$ and $ r\overline{R}$ are isomorphic as $R$-modules. Thus $\overline{R}$ can be generated  by two elements as an $R$-module. 

(2) $\Rightarrow$ (3) Since  $\overline{R}$ is a finitely generated $R$-submodule of $\Q(R)$ and $M$ is a regular ideal, the ideal $(R:\overline{R})$  contains a nonzerodivisor  $m \in M$.  Also, since $R$ is a generalized local ring, $\bigcap_{i} m^iR = \bigcap_{i} M^i  = 0$. 
We claim that $R$ is reduced. 
The argument  is essentially that of \cite[pp.~263-264]{Ma}. Let $r$ be a  nilpotent element of $R$. 
For each $i>0$,  $r/m^i \in \overline{R}$.  
Since $m \in (R:\overline{R})$, we have  $m(r/m^i) \in R$ for all $i >0$.  Thus $ r \in \bigcap_{i>1} m^{i-1}R  = 0$, proving that $r =0$ and $R$ is reduced.  
Since the maximal ideal of $R$ is finitely generated, $R$ is a  Noetherian ring by Proposition~\ref{2.2}(1). Therefore, since $R$ is a reduced one-dimensional  Noetherian local ring with regular maximal ideal,  $\overline{R}$ is a Noetherian ring \cite[Theorem 3]{AH}.

(3) $\Rightarrow$ (1)  Given (3), to prove that $R$ is a Bass ring it suffices by Corollary~\ref{reduced cor} to show that   $R$ is a reduced stable ring. For this it is enough by Theorem~\ref{big} to prove that $R \subseteq \overline{R}$ is a quadratic extension and $\overline{R}$ is a reduced Dedekind ring with at most two maximal ideals. 
Since the $R$-module $\overline{R}$ can be generated by two elements,   $\overline{R}/(\Jac \overline{R})$ has dimension at most $2$ as an $R/M$-vector space. Thus $\overline{R}$ has at most 2 maximal ideals. 
  Since $R$ is a one-dimensional local Noetherian ring with regular maximal ideal  and $\overline{R}$  is a Noetherian ring, $\overline{R}$ is a 
  finite product of Dedekind domains \cite[Theorem 3]{AH}. Hence $\overline{R}$ is a reduced Dedekind ring. By Lemma~\ref{is quad}, $R \subseteq \overline{R}$ is a quadratic extension, and so 
  the proof is complete.
%
\end{proof}



\medskip 

\textsc{Acknowledgment.} I am grateful to the referee for a careful reading of the paper and many suggestions that have improved the presentation.

\end{document}